\documentclass{amsart}
\usepackage{rotating}
\usepackage{enumerate}
\usepackage{amsfonts}
\usepackage{amsmath}
\usepackage{amssymb}
\usepackage{textcomp}
\usepackage{float,lscape}
\usepackage[normalem]{ulem}

\usepackage{enumitem}
\usepackage[arrow, matrix, curve, graph]{xy}
\usepackage{tikz}
\usetikzlibrary{arrows}
\usepackage[hyphens]{url}  

\newtheorem{theorem}{Theorem}[section]

\newtheorem{lemma}[theorem]{Lemma}

\newtheorem{proposition}[theorem]{Proposition}

\newtheorem{corollary}[theorem]{Corollary}

\theoremstyle{definition}

\newtheorem{definition}[theorem]{Definition}

\newtheorem{problem}[theorem]{Problem}

\newtheorem{example}[theorem]{Example}

\newtheorem{remark}[theorem]{Remark}

\newtheorem *{Theorem 1}{Theorem 1}
\newtheorem *{Theorem 2}{Theorem 2}
\newtheorem *{Theorem 5}{Theorem 5}
\newtheorem *{Theorem 3}{Theorem 3}
\newtheorem *{Theorem 4}{Theorem 4}

\newtheorem *{Problem1}{The group ring isomorphism problem [GRIP]}
\newtheorem *{Problem2}{The integral group ring isomorphism problem [IGRIP]}
\newtheorem *{Problem3}{The twisted group ring isomorphism problem (TGRIP)}

\newcommand{\C}{{\mathbb C}}

\newcommand{\Q}{{\mathbb Q}}
\newcommand{\RNum}[1]{\uppercase\expandafter{\romannumeral #1\relax}}
\newcommand{\rNum}[1]{\lowercase\expandafter{\romannumeral #1\relax}}
\newcommand{\rk}{{\operatorname{rk}}}
\newcommand{\Tra}{{\operatorname{Tra}}}
\newcommand{\Ext}{{\operatorname{Ext}}}
\newcommand{\Hom}{{\operatorname{Hom}}}
\newcommand{\Gal}{{\operatorname{Gal}}}

\begin{document}
\title{The twisted group ring isomorphism problem over fields}
\address{Departament of Mathematics, Free University of Brussels, 1050 Brussels, Belgium}
\author{Leo Margolis}
\email{leo.margolis@vub.be}
\author{Ofir Schnabel}
\address{Department of Mathematics, ORT Braude College, 2161002 Karmiel, Israel}
\email{os2519@yahoo.com}

\thanks{The first author is a postdoctoral researcher of the Research Foundation Flanders (FWO - Vlaanderen).
We are grateful for the Technion - Israel Institute of Technology,
for supporting the first author's visit to Haifa}

\begin{abstract}
Similarly to how the classical group ring isomorphism problem
asks, for a commutative ring $R$, which information about a finite
group $G$ is encoded in the group ring $RG$, the twisted group
ring isomorphism problem asks which information about $G$ is
encoded in all the twisted group rings of $G$ over $R$.

We investigate this problem over fields. We start with abelian groups and
show how the results depends on the roots of unity in $R$. In
order to deal with non-abelian groups we construct a
generalization of a Schur cover which exists also when $R$ is not
an algebraically closed field, but still linearizes all projective
representations of a group. We then show that groups from the
celebrated example of Everett Dade which have isomorphic group
algebras over any field can be distinguished by their twisted
group algebras over finite fields.
\end{abstract}

\maketitle

\section{Introduction}\label{Intro}\pagenumbering{arabic} \setcounter{page}{1}
In \cite{MargolisSchnabel} we proposed a twisted version of the
celebrated group ring isomorphism problem (GRIP), namely ``the
twisted group ring isomorphism problem''(TGRIP).

Recall that for a finite group $G$ and a commutative ring $R$, the
group ring isomorphism problem asks whether the ring structure of $RG$
determines $G$ up to isomorphism. In other words, does the existence of a ring isomorphism $RG\cong
RH$ imply the existence of a group isomorphism $G\cong H$, for given groups $G$ and $H$?
Roughly speaking the twisted group
ring isomorphism problem asks if for a group $G$ and a commutative
ring $R$, the ring structure of all the twisted group rings of $G$
over $R$ determines the group $G$.
The role twisted group rings of $G$ over $R$ play for the projective representation theory is in
many ways the same played by the group ring $RG$ for the representation theory of $G$ over $R$,
as it was shown in the ground laying work of I. Schur \cite{Schur}. In this sense the (TGRIP) can also be understood as a question on how strongly the projective representation theory of a group influences its structure.
For results on the classical (GRIP) see \cite{RoggenkampScott, Sehgal1993, Hertweck} for
the case $R = \mathbb{Z}$. Also questions on character degrees, as
addressed e.g. in \cite{Isaacs, Navarro}, can be viewed as results
for the case $R = \mathbb{C}$.

We denote by $R^*$ the unit group in a ring $R$. For a $2$-cocycle $\alpha \in Z^2(G, R^*)$ the twisted group ring $R^\alpha G$ of $G$ over $R$
 with respect to $\alpha$ is the free $R$-module with basis $\{u_g\}_{g \in G}$ where the multiplication on the basis is defined via
\[u_g u_h = \alpha(g, h) u_{gh} \ \ \text{for all} \ \ g,h \in G \]
and any $u_g$ commutes with the elements of $R$. Notice that if we
consider $\alpha$ only as a function (not necessarily a $2$-cocycle) from
$G\times G$ to  $R^*$, then $R^\alpha G$ is associative if
and only if $\alpha$ is a $2$-cocycle, i.e. 
$$\alpha(g,h)\alpha(gh,k) = \alpha(g,hk)\alpha(h,k) \ \ \text{for all} \ \ g,h,k \in G.$$
The ring structure of
$R^{\alpha}G$ depends only on the cohomology class $[\alpha]\in
H^2(G,R^*)$ of $\alpha$ and not on the particular $2$-cocycle.
Notice that the ring $R$ is central in the twisted group ring
$R^{\alpha}G$ and correspondingly the associated second cohomology
group is with respect to a trivial action of $G$ on $R^*$. See
\cite[Chapter 3]{KarpilovskyProjective} for details.

Let $G$ and $H$ be groups and let $R$ be a commutative ring. We define
an equivalence relation which corresponds to the regular (GRIP) by
$G \Delta_R H$ if and only if $RG \cong RH$, and the twisted
problem is defined using a refinement of this relation as follows.

\begin{definition}
Let $R$ be a commutative ring and let $G$ and $H$ be finite groups. We say that $G \sim_R H$ if there exists a group isomorphism
$$\psi :H^2(G,R^*)\rightarrow H^2(H,R^*)$$
such that for any $[\alpha] \in H^2(G,R^*)$ there is a ring isomorphism
$$R^{\alpha}G\cong R^{\psi (\alpha)}H.$$
\end{definition}

It is easy to see that $\sim_R$ is indeed a refinement of
$\Delta_R$, cf. Corollary~\ref{cor:1dimtrivcoho}. The main problem
we are interested in is the following.

\begin{Problem3}
For a given commutative ring $R$, determine the $\sim_R$-classes.
Answer in particular, for which groups $G \sim_R H$ implies $G \cong H$.
\end{Problem3}
In \cite{MargolisSchnabel} we investigated (TGRIP) over the
complex numbers and gave some results for families of groups, e.g.
abelian groups, $p$-groups, groups of central type and groups of
cardinality $p^4$ and $p^2q^2$ for $p,q$ primes. In this paper we
investigate (TGRIP) and related problems for fields other than
$\mathbb{C}$. In particular, our main motivation is to explore:
\begin{enumerate}
  \item The differences between the (TGRIP) and the (GRIP).
  \item The differences between the (TGRIP) over $\mathbb{C}$ and the (TGRIP) over other fields.
\end{enumerate}

For example we showed in \cite[Lemma 1.2]{MargolisSchnabel} that any abelian
group is a $\sim_{\mathbb{C}}$-singleton which is clearly not true
for $\Delta_{\mathbb{C}}$. We show that
over other fields $F$, abelian groups are no longer necessarily
$\sim_{F}$-singletons (see Example~\ref{ex:C9}). This is particularly interesting since,
when $\text{char}(F)$ does not divide $|G|$, i.e. the semi-simple
case, $G \Delta_{F} H$ implies $G \Delta_{\mathbb{C}} H$, while we
show that $G \sim_{F} H$ not necessarily implies $G
\sim_{\mathbb{C}} H$. In this sense, $\mathbb{C}$ is no longer
``the worst" field in distinguishing between groups in the
semi-simple case.

A main result is related to the so called Dade's Example. In
\cite{Dade} E. Dade gave a family of examples of non-isomorphic
groups $G$ and $H$ of order $p^3q^6$ for $p,q$ primes satisfying some
arithmetic conditions, such that $FG \cong FH$ for any field
$F$ while $\mathbb{Z}G\not \cong \mathbb{Z}H$. Consequently, the
ring structure of all the group rings of a group over all fields
is not sufficient to determine the group up to isomorphism.
We prove:

\begin{Theorem 1}
Let $G$ and $H$ be the groups from Dade's example of even order. Then
there exists an infinite number of fields $F$ such that $G \not
\sim_{F} H$.
\end{Theorem 1}

A key ingredient in the proof of Theorem 1, and in general for
studying the ring structure of twisted group rings over fields, is
a generalization of the Schur cover which we develop in Section
\ref{UnSchur}. This generalization exists also when the field is
not algebraically closed. The idea for this kind of cover was
introduced originally by Yamazaki \cite{YamazakiUnschur} and for
this reason we call it a \textit{Yamazaki cover}. This object
generalizes the Schur cover of a group $G$ in the sense that over
not necessarily algebraically closed fields, any projective
representation of $G$ is projectively equivalent to a linear
representation of its Yamazaki cover.

In Theorem~\ref{th:YamazakiGT} we give a group theoretical criterion how a Yamazaki cover of a group can be recognized.
This mimics the well known group theoretical criterion to recognize a Schur cover, but additional properties need to be checked for the Yamazaki cover.
After the construction of Yamzaki covers for both groups from Dade's example we prove Theorem 1.

As mentioned above, a Yamazaki cover may exist when the field $F$
is not necessarily algebraically closed. Throughout this paper,
for a finite group $G$ and a field $F$, we will assume that
$H^2(G,F^*)\cong H^2(G,t(F^*))$. It turns out that this is a
sufficient (and necessary) condition for the existence of a Yamazaki cover of $G$
over $F$. Here, $t(F^*)$ denotes the torsion subgroup of $F^*$. It was shown by Yamazaki that this condition is equivalent to $F^* = (F^*)^{\exp(G/G')}t(F^*)$ \cite[Proposition 3]{YamazakiUnschur}.
For example, for any finite group $G$, the field $F$ can be the complex numbers, the real numbers or any finite field. However, for any
non-trivial $G$ we cannot choose $F$ to be the rational numbers.

The following problem is natural in view of Theorem 1.
\begin{problem}
Let $G$ and $H$ be groups such that $G \sim_{F} H$ for all fields
$F$.
\begin{enumerate}
  \item Is it true that $G$ and $H$ are necessarily isomorphic?
  \item Find families of groups for which the answer to the question above is positive.
\end{enumerate}
\end{problem}

An example of such a family are abelian groups. In fact, if
two abelian groups $G$ and $H$ satisfy $\mathbb{C}G\cong
\mathbb{C}H$ and $H^2(G,\mathbb{C}^*)\cong H^2(H,\mathbb{C}^*)$
then $G \cong H$ (see \cite[Lemma 1.2]{MargolisSchnabel}).
Moreover, it is clear that 
$$G\cong H\Rightarrow G\sim_{F}H\Rightarrow FG\cong FH \text{ and } H^2(G,F^*)\cong H^2(H,F^*)$$
and we have shown \cite{MargolisSchnabel} that in general the converse implications are not true. It is natural to ask if for abelian groups the converse implications are true, and if not, which other conditions we can impose on the field such that they will be true.
In the following theorem we give an answer to this question.  

\begin{Theorem 2}
\begin{enumerate}
\item Let $G$ and $H$ be abelian groups and let $e$ be the exponent of $G$. For any positive integer $n$ let $\zeta _n$ be a fixed primitive $n$-th root of unity in $\mathbb{C}$. Then $G \sim_F H$ implies $G \cong H$ for any field $F$ of characteristic $0$ such that:
\begin{itemize}
\item If $p$ is an odd prime divisor of $e$, then $F$ contains $\zeta_p$ or the inclusion $F \cap \Q(\zeta_{p^2}) \subseteq \Q(\zeta_p)$ holds.
\item If $2$ is a divisor of $e$, then $F$ contains $\zeta_4$ or $F \cap \Q(\zeta_8) = \Q$.
\end{itemize}
However,
\item There
exist non-isomorphic abelian groups $G$ and $H$ and a finite field
$F$ such that $FG$ is semisimple and $G \sim_{F} H$.
   In particular, $\text{char}(F) \nmid |G|$ does not imply that $\sim_{F}$ is a refinement of $\sim_{\mathbb{C}}$.
  \item  There exist abelian groups $G$ and $H$ and a finite field $F$ such that $FG \cong FH$ and
  $H^2(G,F^*) \cong H^2(H,F^*)$, but $G \not \sim_{F} H$.
\end{enumerate}
\end{Theorem 2}

The paper is organized as follows. Most of Section \ref{pre} is
devoted to well-known definitions and tools related to twisted
group rings and the second cohomology group of a finite group.
However, we also prove in Proposition~\ref{prop:abcomesfromEXT} an
interesting result about simple commutative components of twisted
group rings. In Section \ref{abeliangroups} we deal with the relation $\sim_F$ for abelian groups. In particular we prove
Theorem 2. In Section \ref{UnSchur} we introduce and construct the
Yamazaki cover of a group which is a generalization of a Schur
cover of a group which exists also when $F$ is not algebraically
closed. Lastly, in Section \ref{Dade} we prove Theorem 1 by
constructing the Yamazaki covers for the groups from Dade's
example and then evaluating their Wedderburn decompositions.

\section{Preliminaries}\label{pre}
In this section we will recall some definitions and tools that
will be useful later on. Recall that throughout this paper we will
assume for a finite group $G$ and a field $F$ that
$H^2(G,F^*)\cong H^2(G,t(F^*))$, although it is sometimes
redundant.

Clearly two main objects that we need to understand in order to
study the (TGRIP) are the ring structure of twisted group rings, and
the structure of the second cohomology group of a finite group.

We use standard group theoretical notation. In particular we denote by $C_n$ a cyclic group of order $n$, by $\circ(g)$ the order of a group element $g$ in a group $G$, by $Z(G)$ the center and by $G'$ the commutator subgroup of $G$, by $\operatorname{exp}(G)$ the exponent of $G$, by $\operatorname{GL}(V)$ the general linear group acting on a vector space $V$ and by $\operatorname{PGL}(V)$  the projective general linear group, i.e. $\operatorname{GL}(V)/Z(\operatorname{GL}(V))$. Moreover for an abelian group $G$ we denote by $\text{rk}(G)$ the rank of $G$, i.e. the minimal number of generators of $G$. We denote by $\mathbb{F}_q$ a finite field of order $q$.

\subsection{Projective representations and twisted group rings}\label{sec:ProjRep}
The theory presented here is standard and can be found e.g. in \cite[Chapter 3]{KarpilovskyProjective}.
A {\it projective representation} of a group $G$ over a field $F$
is a map
$$\eta: G\rightarrow  GL(V),$$
where $V$ is an $F$-vector space, such that the composition of
$\eta$ with the natural projection from $GL(V)$ to
$PGL(V)$ is a group homomorphism.
 As in the ordinary case, two projective representations are equivalent if
 they differ by a basis change of $V$.
 A projective representation
$\eta: G\rightarrow  GL(V)$ is {\it irreducible} if $V$ admits no
proper $G$-subspace. Two projective representations $\eta_1: G \rightarrow GL(V_1)$ and $\eta_2: G \rightarrow GL(V_2)$ are called \textit{projectively equivalent} if there is a map $\mu: G \rightarrow F^*$ satisfying $\mu(1) = 1$ and a vector space isomorphism $f:V_1 \rightarrow V_2$ such that
$$\eta_1(g) = \mu(g) f^{-1} \eta_2(g) f $$
for every $g \in G$.

With the above notation, we can define $\alpha \in Z^2(G,F^*)$ by
$$\alpha(g_1,g_2)=\eta (g_1) \eta(g_2) \eta (g_1g_2)^{-1},$$
and refer to $\eta$ as an $\alpha$-representation of $G$. For a
fixed 2-cocycle $\alpha$, the set of projective equivalence
classes of irreducible $\alpha$-representations
of $G$ is denoted by $\text{Irr}(G,\alpha)$. As in the
ordinary case, there is a natural correspondence between
projective representations of $G$ over $F$ with an associated
2-cocycle $[\alpha]$, and $F^\alpha G$-modules.

A projective representation $\eta: G\rightarrow GL(V)$ can be extended to a homomorphism of algebras
\begin{equation*}
\begin{array}{rcl}
 \tilde{\eta}: F^\alpha G & \rightarrow &\operatorname{End}_F(V)\\
 \sum_{g\in G} a_g u_g & \mapsto &\sum_{g\in G} a_g \eta (g).
\end{array}
\end{equation*}
For any ring $R$ and an irreducible $R$-module $M$, there is a
surjective ring homomorphism $R\rightarrow \operatorname{End}_D M$ for
$D=\operatorname{End}_R M$. A generalized Maschke's theorem states that if
$\text{char}(F) \nmid |G|$ then any twisted group algebra $F^{\alpha}G$
is semisimple. Therefore, with the above notations for any
irreducible $\alpha$-representation $V$ of $G$, the ring $\operatorname{End}_D V$ can be
identified with one of the components of the Artin-Wedderburn
decomposition of the semisimple algebra $F^{\alpha}G$. In other
words, $F^{\alpha}G$ admits a decomposition
\begin{equation*}
F^{\alpha}G=\bigoplus_{[W]\in\text{Irr}(G,\alpha)}\operatorname{End}_{D_W}(W),
\end{equation*}
where $D_W=\operatorname{End}_{F^{\alpha}G} W$. In particular, if $F$ is a
finite field such that $\text{char}(F) \nmid |G|$ then
\begin{equation*}
F^{\alpha}G=\bigoplus_{[W]\in\text{Irr}(G,\alpha)}\operatorname{End}_{F_W}(W),
\end{equation*}
where here $F_W$ is a field extension of $F$ corresponding to $W$.

In some of our examples later on we will use the structure of the
center of a twisted group algebra. Let $G$ be a finite group and
let $\alpha \in Z^2(G,F^*)$. An element $g\in G$ is called
{\it $\alpha$-regular} if $\alpha(g,h)=\alpha(h,g)$ for any $h\in G$
which commutes with $g$. Note that if $g$ is $\alpha$-regular and
$\beta \in Z^2(G, F^*)$ such that $[\alpha] = [\beta]$ in
$H^2(G,F^*)$  then $g$ is also $\beta$-regular. The following is
well known (see e.g \cite[Theorem 2.4]{NauwelaertsVanOystayen}).
\begin{lemma}\label{lemma:centeroftga}
Let $G$ be a finite group, let $\alpha \in Z^2(G,F^*)$, let $g\in
G$ be an $\alpha$-regular element and let $T$ be a transversal of
the centralizer of $g$ in $G$. Then
\begin{enumerate}
\item The element
$$S_g=\sum _{t\in T} u_tu_gu_t^{-1}$$
is a central element in $F^{\alpha}G$.
\item The elements $S_g$,
where $g$ runs over all the $\alpha$-regular conjugacy classes in
$G$, form an $F$-basis for the center of $F^{\alpha}G$.
\end{enumerate}
\end{lemma}

\subsection{The second cohomology group of a finite group}
The second cohomology group of a group $G$ over the complex
numbers in denoted by $M(G)$ and is called the \textit{Schur multiplier}.
An important tool to understand $H^2(G,F^*)$ is the following
exact sequence (see \cite[Theorem 11.5.2]{KarpilovskyVolII})
\begin{equation}\label{eq:UCT}
1\rightarrow \operatorname{Ext}(G/G',F^*) \rightarrow
H^2(G,F^*)\rightarrow \operatorname{Hom}(M(G),F^*)\rightarrow 1.
\end{equation}
Moreover, this sequence splits (not canonically). Here, for
abelian groups $G,A$
$$\text{Ext}(G,A)=\{[\alpha]\in H^2(G,A) \  | \ \alpha \text{ is symmetric}\},$$
where a cocycle $\alpha \in Z^2(G,A)$ is called \textit{symmetric} if
$\alpha (x,y)=\alpha (y,x)$ for all $x,y\in G$ (see \cite[Chapter 2, \S
1]{KarpilovskyProjective}). Notice that $\operatorname{Ext}(G,A)$ corresponds to
equivalence classes of abelian central extensions of a group $G$
by a group $A$. The map in~\eqref{eq:UCT} from $\operatorname{Ext}(G/G',F^*)$ to
$H^2(G,F^*)$ is the restriction of the inflation map hereby
explained. Let $G$ be a finite group with normal subgroup $N$, let
$A$ be an abelian group and let $\varphi: G\rightarrow G/N$ be the
quotient map. Then, for any $\beta \in Z^2(G/N,A)$ we can define
$\alpha \in Z^2(G,A)$ by
$$\alpha (x,y)=\beta(\varphi (x),\varphi (y)).$$
The map from $Z^2(G/N,A)$ to $Z^2(G,A)$ sending $\beta$ to
$\alpha$ induces a map
$$\text{inf}:H^2(G/N,A)\rightarrow H^2(G,A)$$
which is called the \textit{inflation map}. The map
in~\eqref{eq:UCT} from $\operatorname{Ext}(G/G',F^*)$ to $H^2(G,F^*)$ is the
restriction to the subgroup $\operatorname{Ext}(G/G',F^*)$ of the inflation map
from $H^2(G/G',F^*)$ to $H^2(G,F^*)$. In the sequel we will
sometimes abuse notations and denote the image of this map in
$H^2(G,F^*)$ as $\operatorname{Ext}(G/G',F^*)$ and its complement
in $H^2(G,F^*)$ by $\operatorname{Hom}(M(G),F^*)$

For the sake of completeness and for later use, before going
forward with the description of the second cohomology group, we
would like to introduce a third map which is associated to the
second cohomology group. Let
$$1\rightarrow
N \rightarrow H\overset{\alpha}\rightarrow G\rightarrow 1$$ be a
central extension, i.e. $N$ is a subgroup contained in the center of $H$ such that $H/N \cong G$.
Let $\mu$ be a section of $\alpha$ and define
$f\in Z^2(G,N)$ by $f(x,y)=\mu (x) \mu (y) \mu (xy)^{-1}$. Then,
for any abelian group $A$ and any $\chi \in \operatorname{Hom}(N,A)$ we have $\chi
\circ f \in Z^2(G,A)$ and the cohomology class $[\chi \circ f]$ does not depend on the
choice of $\mu$.

\begin{definition}\label{def:Tra}
With the above notation, the map $\operatorname{Tra}: \operatorname{Hom}(N,A)\rightarrow
H^2(G,A)$ defined by $\chi \mapsto [\chi \circ f]$ is called the
\emph{transgression map}.
\end{definition}

We like to point out that the three maps mentioned above,
inflation, restriction and transgression, are connected to each
other as demonstrated in the celebrated Hochschild and Serre exact
sequence.

Now recall that (see e.g. \cite[Corollary 2.3.17]{KarpilovskyProjective}) for any natural numbers $n_1$,...,$n_r$
\begin{equation}\label{eq:EXTdecomposition}
\operatorname{Ext}(\Pi _{i=1}^rC_{n_r}, F^*) \cong \Pi _{i=1}^r
\operatorname{Ext}(C_{n_r}, F^*).
\end{equation}
Therefore, in order to understand $\operatorname{Ext}(G/G',F^*)$
it is sufficient to understand the description of
$\operatorname{Ext}(C_n,F^*)\cong H^2(C_n,F^*)$. This is well
known (see e.g. \cite[Theorem 1.3.1]{KarpilovskyProjective}):
\begin{equation}\label{eq:cohoofcyclic}
\operatorname{Ext}(C_n,F^*)\cong H^2(C_n,F^*)\cong F^*/(F^*)^n.
\end{equation}
Notice that by our assumption that always $H^2(G,F^*)\cong
H^2(G,t(F^*))$, we deduce that $H^2(C_n,F^*)\cong F^*/(F^*)^n\cong
t(F^*)/t(F^*)^n$. This is a finite cyclic group for any field $F$
as any two elements $a,b \in t(F^*)$ generate a finite, and hence cyclic, group and so also $\langle a, b \rangle / \langle a,b \rangle^n$ is cyclic.

We will use the above to recall the known structure of the second
cohomology group of abelian groups (see e.g. \cite[Corollary in \S
2.2]{YamazakiAbelian}).

Let $G$ be an abelian group. Then $G$ admits a decomposition
\begin{equation}\label{eq:decompofabeliangroupptok}
G \cong C_{n_1}\times C_{n_2}\times \ldots \times C_{n_r} = \langle
x_1 \rangle \times \langle x_2 \rangle\times \ldots \times \langle
x_r \rangle
\end{equation}
such that $n_i$ is a divisor of $n_{i+1}$ for any $1\leq i\leq
r-1$. Clearly,
\begin{equation}\label{eq:STforAB}
\operatorname{Ext}(G/G',F^*) \cong \prod _{i=1}^r F^*/(F^*)^{n_i}.
\end{equation}
We want to describe $\operatorname{Hom}(M(G),F^*)$. First notice,
that if $g$ and $h$ are commuting elements in a group $G$ with
orders $n$ and $m$ correspondingly, then $[u_g,u_h]=\lambda$ in
the twisted group algebra $ F^{\alpha}G$, and $\lambda$ is a root
of unity dividing $\gcd(m,n)$. This follows directly from the fact
that for any $x\in G$ the element $u_x^{\circ (x)}$ is central in $
F^{\alpha}G$ and therefore $[u_g^{\circ (g)},u_h]=\lambda ^{\circ
(g)}=1$. Now, for any natural numbers $n$ and $m$ denote by
$d(n,m,F)$ the maximal order of a root of unity in $F$ which
divides the greatest common divisor of $m$ and $n$. If $m$ is a
divisor of $n$, we denote $d(n,m,F)$ by $d(m,F)$. By the above,
for $G$ as in~\eqref{eq:decompofabeliangroupptok},
\begin{equation}\label{eq:NAforA}
\operatorname{Hom}(M(G),F^*)\cong \prod
_{i=1}^{r-1}C_{d(n_i,F)}^{r-i},
\end{equation}
generated by the tuple of functions
$$\left(\alpha _{ij}\right)_{1\leq i<j\leq r},$$
where $\alpha _{ij}(x_i,x_j)$ is a primitive $d(n_i,F)$-th root of
unity and $1$ elsewhere. From~\eqref{eq:UCT},~\eqref{eq:STforAB}
and~\eqref{eq:NAforA}, for $G$ as
in~\eqref{eq:decompofabeliangroupptok} we have
\begin{equation}\label{eq:cohomologyofABgroup}
H^2(G,F^*)\cong \left(\prod _{i=1}^r F^*/(F^*)^{n_i}\right) \times
\left(\prod _{i=1}^{r-1}C_{d(n_i,F)}^{r-i}\right).
\end{equation}
As a consequence of the above, over the complex numbers, non-isomorphic abelian groups of the same cardinality
admit non-isomorphic cohomology groups (see \cite{Schur} or \cite[Corollary 2.3.16]{KarpilovskyProjective}).

\subsection{Commutative components of twisted group rings}
In this section we study twisted group rings admitting a
commutative component in their Wedderburn decomposition. We start with a straightforward result.

\begin{lemma}\label{lemma:GRNOTGR}
Let $G$ be a group, $R$ a commutative ring and let $\alpha \in Z^2(G,R^*)$. If there exists an $\alpha$-projective representation of dimension $1$, then $\alpha$ is cohomologicaly trivial.
\end{lemma}
\begin{proof}
This is clear by the definition of co-boundary.
\end{proof}

\begin{corollary}\label{cor:1dimtrivcoho}
Let $G$ be a group, let $R$ be a commutative ring and let
$\alpha \in Z^2(G,R^*)$. Then $R^{\alpha}G$ admits a
$1$-dimensional simple module if and only if $\alpha$ is
cohomologically trivial. In particular, $\sim_R$ is a refinement
of $\Delta_R$.
\end{corollary}
We wish to generalize this result to commutative components with
dimension not necessarily $1$ over fields.
\begin{proposition}\label{prop:abcomesfromEXT}
Let $G$ be a group, let $F$ be a field such that $\text{char}(F) \nmid |G|$ and let $[\alpha]\in H^2(G,F^*)$. Then $F^{\alpha}G$ admits a commutative simple component in its Wedderburn decomposition if and only if
$[\alpha]$ is in the image of the inflation map from $\operatorname{Ext}(G/G',F^*)$ to $H^2(G,F^*)$ as defined in Section~\ref{sec:ProjRep}.
\end{proposition}
\begin{proof}
Denote by $\bar{F}$ the algebraic closure of $F$.
Consider the following commutative diagram related to the exact sequence in \eqref{eq:UCT}. Here the vertical maps are just obtained by understanding elements of $Z^2(G, F^*)$ as elements of $Z^2(G, \bar{F}^*)$.
\begin{center}
\begin{tikzpicture}
\tikzset{thick arc/.style={->, black, fill=none,  >=stealth,
text=black}} \tikzset{node distance=2cm, auto}
 \node (triv1F){$1$};
 \node (EXT_F) [right of=triv1F] {$\operatorname{Ext}(G/G',F^*)$};
  \tikzset{node distance=3cm, auto}
 \node (H2F) [right of=EXT_F] {$H^2(G,F^*)$};
  \node (HOM_F) [right of=H2F] {$\operatorname{Hom}(M(G),F^*)$};
   \tikzset{node distance=2cm, auto}
  \node (triv2F)[right of=HOM_F]{$1$};
 \tikzset{node distance=2cm, auto}
   \node (triv1Fclosed)[below of=triv1F] {$1$};
  \node (EXT_Fclosed) [right of=triv1Fclosed] {$\operatorname{Ext}(G/G',\bar{F}^*)$};
   \tikzset{node distance=3cm, auto}
 \node (H2Fclosed) [right of=EXT_Fclosed] {$H^2(G,\bar{F}^*)$};
  \node (HOM_Fclosed) [right of=H2Fclosed] {$\operatorname{Hom}(M(G),\bar{F}^*)$};
   \tikzset{node distance=2cm, auto}
  \node (triv2Fclosed)[right of=HOM_Fclosed]{$1$};

\draw[thick arc, draw=black] (triv1F) to node [above] {$$} (EXT_F);
\draw[thick arc, draw=black] (EXT_F) to node [above] {$\text{inf}$} (H2F);
\draw[thick arc, draw=black] (H2F) to node [above] {$d$} (HOM_F);
\draw[thick arc, draw=black] (HOM_F) to node [above] {$$} (triv2F);

\draw[thick arc, draw=black] (triv1Fclosed) to node [above] {$$} (EXT_Fclosed);
\draw[thick arc, draw=black] (EXT_Fclosed) to node [above] {$\text{inf}$} (H2Fclosed);
\draw[thick arc, draw=black] (H2Fclosed) to node [above] {$d$} (HOM_Fclosed);
\draw[thick arc, draw=black] (HOM_Fclosed) to node [above] {$$} (triv2Fclosed);

\draw[thick arc, draw=black] (EXT_F) to node [left] {$ $} (EXT_Fclosed);
\draw[thick arc, draw=black] (H2F) to node [left] {$ $} (H2Fclosed);
\draw[right hook->, draw=black] (HOM_F) to node [left] {$ $} (HOM_Fclosed);

\end{tikzpicture}
\end{center}
Assume first that $[\alpha]$ is in the image of the inflation map from $\operatorname{Ext}(G/G',F^*)$ to $H^2(G,F^*)$ and denote its (unique) pre image
in $\operatorname{Ext}(G/G',F^*)$ by $[\beta]$. Then, since $\operatorname{Ext}(G/G',\bar{F}^*)$ is trivial, $[\beta]$ is also trivial as an element of $\operatorname{Ext}(G/G',\bar{F}^*)$ and therefore
$\gamma :=\text{inf}( [\beta])$ is the trivial cohomology class in $H^2(G,\bar{F}^*)$. Hence $\bar{F}^{\gamma}G\cong \bar{F}G$
admits $\bar{F}$ as a simple component. Now, since $\bar{F}^{\gamma}G\cong F^{\alpha }G\otimes _F \bar{F}$ we conclude that $F^{\alpha }G$ admits a commutative simple component.

Conversely, assume that $F^{\alpha }G$ admits a commutative simple component. Let $[\gamma]$ be the cohomology class in $H^2(G, \bar{F}^*)$ obtained from $[\alpha]$. Then, $F^{\alpha }G \otimes _F \bar{F}\cong \bar{F}^{\gamma}G$
also admits a commutative simple component. However, since $\bar{F}$ is algebraically closed this component is $\bar{F}$ itself.
Consequently, by Corollary~\ref{cor:1dimtrivcoho} $[\gamma]$ is the trivial cohomology class.
Clearly from the diagram above $[\alpha]$ is in the image of the inflation map from $\operatorname{Ext}(G/G',F^*)$ to $H^2(G,F^*)$
\end{proof}

\section{Abelian groups}\label{abeliangroups}
In this section we will not assume that $H^2(G, F^*) = H^2(G, t(F^*))$, i.e. our results are valid for all fields.

The main results of this section is Theorem $2$ and Theorem~\ref{th:AbelianIso} . The proof of Theorem 2 is done
in three steps. In Theorem~\ref{th:Char0} we prove Theorem
2(1), Example~\ref{ex:C9} and Example~\ref{ex:OddOverReal} shows Theorem 2(2) and lastly,
Proposition~\ref{prop:notinrelation} gives Theorem 2(3). 

In a way, the group ring isomorphism problem asks whether it is
possible to distinguish groups by their group ring structure over
a commutative ring $R$. For this purpose it is clear that the ring
of integers is ``the best'' ring since for any commutative ring
$R$ and finite groups $G$ and $H$ the isomorphism
$\mathbb{Z}G\cong \mathbb{Z}H$ implies that $RG \cong RH$. Also,
in a sense, in the semi-simple case, the field of complex numbers
is ``the worst'' commutative domain in the sense that if $F$ is a
commutative domain, $G$ and $H$ are finite groups such that
$FG\cong FH$ is semi-simple then $\mathbb{C}G\cong \mathbb{C}H$.
This follows from the fact that if $\bar{F}$ denotes the algebraic
closure of the quotient field of $F$ then $\bar{F}G \cong \bar{F}
\otimes_F FG$ and the character theories over algebraically closed
fields coincide in the semi-simple case \cite[Corollary
18.11]{CR1}. We don't know yet, if $\mathbb{Z}$ is also ``best''
in distinguishing groups in the twisted case, but it is clear that
$\mathbb{C}$ is no longer the ``worst'' in the semi-simple case. We give two simple examples:

\begin{example}\label{ex:C9}
  Let $G \cong C_3\times C_3$, let $H \cong C_9$ and let $F = \mathbb{F}_{17}$.
  Then, $H^2(G,F^*)$ and $H^2(H,F^*)$ are trivial and
  $$ FG\cong FH\cong F\oplus 4\mathbb{F}_{17^2}.$$
  So $G\sim _F H$.

  As $G \not\cong H$ it is clear that $G \not \sim _{\mathbb{C}} H$, since these groups admit non-isomorphic Schur multipliers by \eqref{eq:cohomologyofABgroup} (see also
  \cite[Lemma 1.2]{MargolisSchnabel}).
\end{example}

\begin{example}\label{ex:OddOverReal}
Let $G$ and $H$ be abelian groups of odd order such that $|G| = |H|$ and denote by $\mathbb{R}$ the real numbers. As $\mathbb{R}^*/(\mathbb{R}^*)^p = 1$ for any odd prime $p$ we conclude that $\Ext(G, \mathbb{R}^*) = \Ext(H, \mathbb{R}^*) = 1$. Moreover, $F$ contains no primitive root of unity of order $p$, for any odd prime $p$. Hence also $\Hom(M(G), \mathbb{R}^*) = \Hom(M(H), \mathbb{R}^*) = 1$ and overall $H^2(G, \mathbb{R}^*) = H^2(H, \mathbb{R}^*) = 1$. 
	 Furthermore, for any non-trivial representation of $G$ or $H$ there is an element of odd order which does not lie in the kernel. Hence the module corresponding to such a representation of $G$ or $H$ is isomorphic to $\mathbb{C}$. So
	\[\mathbb{R} G \cong \mathbb{R} H \cong \mathbb{R} \oplus \frac{|G|-1}{2} \mathbb{C}. \]
Hence $G \sim_F H$.

Also here we know that $G \not\cong H$ and hence $G \not \sim _{\mathbb{C}} H$, as these groups admit non-isomorphic Schur multipliers by \eqref{eq:cohomologyofABgroup}.
\end{example}

Notice, that for abelian groups $G$ and $H$, if $\mathbb{C}G\cong
\mathbb{C}H$ and $M(G)\cong M(H)$ then $G$ and $H$ are isomorphic. In particular, if $G \sim_{\mathbb{C}} H$ then $G\cong H$.
By the above examples, this is not true in general over other fields. Next we will search for conditions on a field $F$ and abelian groups $G$ and $H$, such that under these conditions $G \sim_{F} H$ will imply $G\cong H$. The following lemma will be key.
\begin{lemma}\label{lemma:exp-m-isomorphic}
Let $G$ and $H$ be finite abelian $p$-groups for a prime $p$. Let $F$ be a field and let $p^m$ be the cardinality of the maximal
$p$-subgroup of $F^*$ (here $m$ being infinity is allowed). If $G \sim_F H$, then the maximal subgroups of $G$ and $H$ of
exponent dividing $p^m$ are isomorphic.
In particular, for $m\geq 1$ the groups $G$ and $H$ have the same rank.
\end{lemma}

\begin{proof}
Note first, that if the charateristic of $F$ equals $p$, then $FG \cong FH$ implies $G \cong H$, as the modular isomorphism problem has a positive solution for abelian groups \cite[Corollary 5]{Passmanp4}. So assume that the characteristic of $F$ is different from $p$.

First, the lemma is clear for $m=0$, that is if $F$ contains no
primitive $p$-th roots of unity. Second, if $F^*$ admits a
$p$-subgroup of infinite order, then $\Ext(G, F^*) = \Ext(H, F^*) = 1$ and hence
$$M(G)\cong H^2(G,F^*)\cong H^2(H,F^*)\cong M(H).$$
Consequently by \cite[Lemma 1.2]{MargolisSchnabel} $G$ and $H$ are
isomorphic. We are left with the case $m$ is some natural number.

Note that by \cite[Proposition 2.11]{TGRIPIII} we know that 
$$\Hom(M(G), F^*)\cong \Hom(M(H), F^*).$$
Assume
\begin{equation*}
G=\langle \sigma _1 \rangle \times \langle \sigma _2 \rangle
\times \ldots \times \langle \sigma _k \rangle,
\end{equation*}
\begin{equation*}
H=\langle \tau _1 \rangle \times \langle \tau _2 \rangle \times
\ldots \times \langle \tau _r \rangle.
\end{equation*}
Now define for $1\leq i \leq m$
\begin{equation*}
\begin{array}{cc}
   b_i(G)=
  \left\{
  \begin{array}{cc}
    |\{ j:\circ (\sigma _j )=p^i \}|, & i<m . \\
    |\{ j:\circ (\sigma _j )\geq p^m\}| , & i=m .
  \end{array}
  \right. &
\end{array}
  b_i(H)=
  \left\{
  \begin{array}{cc}
    |\{ j:\circ (\tau _j )=p^i\}| , & i<m . \\
    |\{ j:\circ (\tau _j )\geq p^m\}| , & i=m .
  \end{array}
  \right.
\end{equation*}
By~\eqref{eq:cohomologyofABgroup} we have
$$\Hom(M(G), F^*) \cong \Hom(M(H), F^*) \cong a_m C_{p^m}\times a_{m-1}C_{p^{m-1}}\times \ldots \times a_1 C_p$$
for some natural numbers $a_1$,...,$a_m$. Also
by~\eqref{eq:cohomologyofABgroup} we can express the $a_i$ in
terms of the $b_i$ such that $a_m$ only depends on $b_m$,
$a_{m-1}$ only depends on $b_m$ and $b_{m-1}$ etc. Namely:
$$\binom{b_i(G)}{2} + b_i(G)\left(\sum_{j = i+1}^m b_j(G)\right) = a_i = \binom{b_i(H)}{2} + b_i(H)\left(\sum_{j = i+1}^m b_j(H)\right).$$
This formula follows as, in the notation of \eqref{eq:cohomologyofABgroup}, the $b_i(G)$ cyclic groups contribute $\binom{b_i(G)}{2}$ copies of cyclic groups of order $C_{p^i}$ to $M(G)$, one for each choice of two such groups, and each cyclic group of order bigger than $p^i$ contributes $b_i(G)$ copies.

Consequently, $b_i(G)=b_i(H)$ for any $1\leq i \leq m$ and the result follows.
\end{proof}

We are now ready do prove
\begin{theorem}\label{th:AbelianIso}
	Let $G$ and $H$ be finite abelian groups and $e$ the exponent of $G$. Let $F$ be a field with the following property: The characteristic of $F$ does not divide the order of $G$ and if $p^n$ is a prime power dividing $e$ such that $F$ contains no primitive root of unity of order $p^n$ and $\zeta$ is a primitive root of unity of order $p^n$ in an extension of $F$, then $F(\zeta)$ contains no primitive $p^{n+1}$-th root of unity.
	
	Then $G \sim_F H$ implies $G \cong H$.
\end{theorem}

\begin{proof}
	Clearly it is sufficient to prove the theorem for abelian
	$p$-groups for primes $p$. Set $e_G = \log_p(\text{exp}(G))$ and
	$e_H = \log_p(\text{exp}(H))$. Let $p^m$ be the cardinality of the
	maximal $p$-subgroup of $F^*$ (here $m$ being infinity is
	allowed). If $m\geq \max \{e_G, e_H\}$ the result follows from Lemma~\ref{lemma:exp-m-isomorphic}.
	
	Assume $m<e_G$, denote by $K$ the prime field of $F$ and by $\zeta_n$ a primitive $n$-th root of unity over $F$ for any positive integer $n$. Set
$$d = \left\{ \begin{array}{cc} \varphi(p^m),& \text{if} \ \ m \geq 1, \\  {[}F \cap K(\zeta_p) : K{]} ,& \text{else} \end{array}\right. $$
where $\varphi$ denotes Euler's totient function.
	Then in the Artin-Wedderburn
	decomposition of $FG$ the maximal field extension appearing is $F(\zeta_{p^{e_G}})$ which has 
	degree $\frac{\varphi(p^{e_G})}{d}$ over $F$ by our assumptions on $F$.  Since $FG\cong FH$,
	the degree of the maximal field extension in the Artin-Wederbrun
	decomposition of $FH$ is also
	$\frac{\varphi(p^{e_G})}{d} =
	\frac{\varphi(p^{e_H})}{d}$. Consequently, $e_G=e_H =:
	e$. Since $FG \cong FC_{p^{e}} \otimes F(G/C_{p^{e}})$ and  $FH
	\cong FC_{p^{e}} \otimes F(G/C_{p^{e}})$ we conclude by induction
	that $G \cong H$.
\end{proof}

For fields of characteristic 0 this can be reformulated more elegantly.  For a positive integer $n$ denote by $\zeta_n$ a fixed primitive $n$-th root of unity in $\C$.

\begin{theorem}\label{th:Char0}
	Let $G$ and $H$ be abelian groups and $e$ the exponent of $G$. Let $F$ be a field of characteristic $0$ such that if $p$ is an odd prime divisor of $e$, then $F$ contains $\zeta_p$ or $F \cap \Q(\zeta_{p^2}) \subseteq \Q(\zeta_p)$, and if $2$ is a divisor of $e$, then $F$ contains $\zeta_4$ or $F \cap \Q(\zeta_8) = \Q$. Then $G \sim_F H$ implies $G \cong H$.
\end{theorem}
\begin{proof}
	Let $p$ be a prime and $n$ a positive integer. Assume $F$ does not contain $\zeta_{p^{n-1}}$ and let $m$ be the maximal integer such that $F$ contains $\zeta_{p^m}$. We will show that Theorem~\ref{th:AbelianIso} is applicable under our assumptions.
	
	First assume $p$ is odd. The Galois group of $\Q(\zeta_{p^n})/\Q$ is isomorphic to $C_{p-1} \times C_{p^{n-1}}$. If $F$ contains $\zeta_p$, then $\Gal(F \cap \Q(\zeta_{p^n})/\Q)$ contains a subgroup isomorphic to $C_{p-1}$ and $A = \Gal(\Q(\zeta_{p^n})/F \cap \Q(\zeta_{p^n}))$ is a cyclic group of $p$-power order, hence uniserial. The extensions of $F$ by $\zeta_{p^{m+1}}$, $\zeta_{p^{m+2}}$,...,$\zeta_{p^n}$ correspond to the unique composition series of $A$. We conclude that $F(\zeta_{p^{n-1}})$ does not contain $\zeta_{p^n}$. On the other hand, if $F \cap \Q(\zeta_{p^2}) \subseteq \Q(\zeta_p)$, then $\Gal(F \cap \Q(\zeta_{p^n})/\Q)$ contains no cyclic group of order $p$ and $\Gal(F(\zeta_{p^{n-1}}) \cap \Q(\zeta_{p^n})/\Q)$ contains no cyclic group of order $p^{n-1}$. So also in this case $F(\zeta_{p^{n-1}})$ does not contain $\zeta_{p^n}$.  
	
	Now assume $p = 2$. Then $\Gal(\Q(\zeta_{2^n})/\Q) \cong C_{2^{n-1}} \times C_2$, where as a generator of the direct factor $C_2$ one can take the complex conjugation. Note that the last fact follows, as the complex conjugation is not the square of any automorphism since the congruence $x^2 \equiv -1 \bmod 2^n$ is not solvable for $n \geq 2$. So if $F$ contains $\zeta_4$, then $\Gal(\Q(\zeta_{2^n})/F \cap \Q(\zeta_{2^n}))$ is uniserial and $F(\zeta_{2^{n-1}})$ does not contain $\zeta_{2^n}$. On the other hand, if $F \cap \Q(\zeta_8) = \Q$, then $\Gal(F(\zeta_{2^{n-1}}) \cap \Q(\zeta_{2^n})/\Q)$ contains no cyclic group of order $2^{n-1}$. So $F(\zeta_{2^{n-1}})$ does not contain $\zeta_{2^n}$.
\end{proof}

\begin{corollary}
Let $G$ and $H$ be abelian groups. Assume $F$ is algebraically closed field or a cyclotomic field. Then $G \sim_F H$ implies $G \cong H$.	
\end{corollary}

\begin{proof}
If the characteristic of $F$ divides the order of $G$ we are done by \cite[Corollary 5]{Passmanp4}, so assume the characteristic of $F$ does not divide $|G|$.
Let $p$ be a prime dividing $|G|$. If $F$ is algebraically closed, then $F$ contains a $p^m$-th root of unity for any positive integer $m$. So it follows from Lemma~\ref{lemma:exp-m-isomorphic} that $G$ and $H$ are isomorphic. So assume, $F = \Q(\zeta_n)$ for some positive integer $n$. If $p \mid n$, then $F$ contains $\zeta_p$ and if $p \nmid n$, then $F \cap \Q(\zeta_{p^2}) = \Q$. Here we use that $\Q(\zeta_n) \cap \Q(\zeta_\ell) = \Q(\zeta_{\text{gcd}(n,\ell)})$ for any positive integer $\ell$. The same argument shows that $F$ contains $\zeta_4$, if $4 \mid n$ and $F \cap \Q(\zeta_8) = \Q$, if $4 \nmid n$.  Hence the result follows from Theorem~\ref{th:Char0}. 
\end{proof}

It is clear for groups $G$ and $H$ and a field $F$ that $$G\cong H\Rightarrow G\sim_{F}H\Rightarrow FG\cong FH \text{ and } H^2(G,F^*)\cong H^2(H,F^*)$$
and we have shown that in general the converse implications are not true. We have also shown that even for abelian groups in general $G \sim_F H$ does not imply $G\cong H$. We next show that there are abelian groups $G$ and $H$ which can have isomorphic group algebras and isomorphic second cohomology groups over some field, but nevertheless do not satisfy $G \sim_F H$ over any field. 
\begin{proposition}\label{prop:notinrelation}
Let $G = C_8\times C_2$ and let $H = C_4\times C_4$. Then
\begin{enumerate}
\item There exist fields $F$ such that $FG\cong FH$ and
$H^2(G,F)\cong H^2(H,F)$. But
\item For any field $F$ the relation
$G \sim_{F} H$ does not hold.
\end{enumerate}
\end{proposition}

\begin{proof}
Let $\mathbb{F}_q$ be a finite field such that $q-1$ is divisible by $2$
but not divisible by $4$, that is $\mathbb{F}_q$ contains roots of unity of
order $2$ but does not contain roots of unity of order $4$. In
this case it follows from \eqref{eq:cohomologyofABgroup} that
$$H^2(G,\mathbb{F}_q^*)\cong H^2(H,\mathbb{F}_q^*)\cong C_2 \times C_2 \times C_2.$$
If additionally $q^2-1$ is divisible by $8$, then
$$\mathbb{F}_q G\cong \mathbb{F}_q H\cong 4\mathbb{F}_q\oplus \mathbb{F}_{q^2}.$$
This concludes the first part of the proposition. We want to show that
for $F$ any field, $G\not \sim _{F} H$.
Let $\text{char}(F)=p$. If $p=2$ then, since the modular
isomorphism problem has a positive solution for abelian groups, $FG \not \cong
FH$ and therefore, $G\not \sim _{F} H$ \cite[Corollary 5]{Passmanp4}.

Consequently we may assume $p \neq 2$. Let
$$G = C_8\times C_2= \langle g_1 \rangle \times \langle g_2 \rangle, \quad H = C_4\times C_4=\langle h_1 \rangle \times \langle h_2 \rangle .$$
In the following arguments about the center of twisted group
algebras we use Lemma~\ref{lemma:centeroftga}. If there exists a
primitive 4-th root of unity $\zeta$ in $F$ then there exists a
twisted group algebra over $H$ with a 1-dimensional center,
determined by the relation $[u_{h_1}, u_{h_2}] = \zeta$. But all
twisted group rings over $G$ admit a center of dimension at least
$4$ spanned by $u_{g_1}^2$. We are left with the case $p \neq 2$ and
$F$ not containing a primitive 4-th root of unity. Consider $[\alpha]\in H^2(H,F^*)$
determined in $F^{\alpha}H$ by
$$[\alpha]:\quad [u_{h_1},u_{h_2}]=-1,\quad u_{h_1}^4=u_{h_2}^4=1.$$
Then the center of $F^{\alpha}H$ is isomorphic to $F(C_2\times
C_2)\cong 4F$ and in particular, $4$ is not a divisor of the order
of any central element of $F^{\alpha}H$. However, for any
$[\beta]\in H^2(G,F^*)$ the element $u_{g_1}^2$ is central in
$F^{\beta}G$ of order multiple of $4$. This completes the proof.
\end{proof}

It is interesting to compare the situation in Proposition~\ref{prop:notinrelation} to the following example.
\begin{example}
Let $G=C_{16}\times C_4$ and let $H=C_8 \times C_8$, then $G\sim
_{\mathbb{F}_{31}} H$.
\end{example}
\begin{proof}
Let $\mathbb{F}_{31}=F$ and $\mathbb{F}_{{31}^2}=K$. Assume
$$G = C_{16}\times C_4= \langle g_1 \rangle \times \langle g_2 \rangle, \quad H = C_8\times C_8=\langle h_1 \rangle \times \langle h_2 \rangle .$$
Since, $F^*$ admits an element of order $2$ but no elements of
order $4$, by~\eqref{eq:cohomologyofABgroup}
$$H^2(G,F^*)\cong H^2(H,F^*)\cong C_2\times C_2\times C_2.$$
In order to prove that $G\sim _{\mathbb{F}_{31}} H$ we will need
also to describe generators for the cohomology groups. For
$H^2(G,F^*)$ we have the generators
$$[\alpha _1]:\hspace{0.5cm}[u_{g_1},u_{g_2}]=-1,\hspace{0.5cm} u_{g_1}^{16}=1,\hspace{0.5cm} u_{g_2}^{4}=1,$$
$$[\alpha _2]:\hspace{0.5cm}[u_{g_1},u_{g_2}]=1,\hspace{0.5cm} u_{g_1}^{16}=-1,\hspace{0.5cm} u_{g_2}^{4}=1,$$
$$[\alpha _3]:\hspace{0.5cm}[u_{g_1},u_{g_2}]=1,\hspace{0.5cm} u_{g_1}^{16}=1,\hspace{0.5cm} u_{g_2}^{4}=-1.$$
For $H^2(H,F^*)$ we have the generators
$$[\beta _1]:\hspace{0.5cm}[u_{h_1},u_{h_2}]=-1,\hspace{0.5cm} u_{h_1}^{8}=1,\hspace{0.5cm} u_{g_2}^{8}=1,$$
$$[\beta _2]:\hspace{0.5cm}[u_{h_1},u_{h_2}]=1,\hspace{0.5cm} u_{h_1}^{8}=-1,\hspace{0.5cm} u_{h_2}^{8}=1,$$
$$[\beta _3]:\hspace{0.5cm}[u_{h_1},u_{h_2}]=1,\hspace{0.5cm} u_{h_1}^{8}=1,\hspace{0.5cm} u_{h_2}^{8}=-1.$$
We claim that the isomorphism from $\psi :H^2(G,F^*)\rightarrow
H^2(H,F^*)$ sending $[\alpha _i]$ to $[\beta _i]$ induces a ring
isomorphism $F^{\alpha}G\cong F^{\psi(\alpha)}H $ for any
$[\alpha] \in H^2(G,F^*)$.
 The group rings $FG$ and $FH$ are clearly isomorphic,
namely to $4F \oplus 30K$. Now, let $[\alpha]\in H^2(G,F^*)$ and
$[\beta]\in H^2(H,F^*)$ be non-trivial cohomology classes such
that $F^{\alpha}G$ and $F^{\beta}H$ are commutative. By
Lemma~\ref{lemma:GRNOTGR} the twisted group rings admit no
$1$-dimensional components (over $F$). And therefore, since $K ^*$
admits elements of order $32$ we conclude that
$$F^{\alpha}G\cong F^{\beta}H \cong \oplus _{i=1}^{32} K.$$

A well known result says that for any group $G$, the order of a cohomology
class $[\gamma]\in H^2(G,F^*)$ divides the dimension of each
$\gamma$-projective representation of $G$ \cite[Proposition 6.2.6]{KarpilovskyProjective}. Therefore, by
Proposition~\ref{prop:abcomesfromEXT} for any $[\alpha]\in
H^2(G,F^*)$ and $[\beta]\in H^2(H,F^*)$ such that $F^{\alpha}G$
and $F^{\beta}H$ are not commutative, they are isomorphic to a
direct sum of $2\times 2$-matrix rings over $F$ and $K$. Therefore
they are isomorphic if and only if their centers are isomorphic.

By Lemma~\ref{lemma:centeroftga} for any $[\alpha]\in H^2(G,F^*)$
and $[\beta]\in H^2(H,F^*)$ such that $F^{\alpha}G$ and
$F^{\beta}H$ are not commutative, the center of $F^{\alpha}G$ is
generated (as an algebra) by $u_{g_1}^2,u_{g_2}^2$ and similarly
the center of $F^{\beta}H$ is generated (as an algebra) by
$u_{h_1}^2,u_{h_2}^2$. Again, by Lemma~\ref{lemma:GRNOTGR}, if the
restriction of $\alpha$ (similarly $\beta$) to the subgroup
generated by $g_1,g_2$ (similarly $h_1,h_2$) is non-trivial then
$$Z(F^{\alpha}G)\cong Z(F^{\beta}H)\cong 8K.$$
This holds for the cohomology classes
$$[\alpha _1 \alpha _2],[\alpha_1 \alpha _3],[\alpha _1 \alpha _2 \alpha _3]\in H^2(G,F^*), \quad
[\beta _1 \beta _2],[\beta_1 \beta _3],[\beta _1 \beta _2
\beta_3]\in H^2(H,F^*).$$ Finally,
$$Z(F^{\alpha_1}G)\cong Z(F^{\beta_1}H)\cong 4F\oplus 2K.$$
This completes the proof.
\end{proof}

\section{The Yamazaki cover}\label{UnSchur}
Let $p$ be prime, let $F$ be a field and let $\zeta$ be a
primitive root of unity of order $p^k$ which is maximal in the
sense that there are no primitive roots of unity in $F$ of order
$p^{k+1}$. Then, by our assumption that $H^2(G,F^*)\cong
H^2(G,t(F^*))$, we may always assume that for a cyclic group
$C_{p^r}$ with generator $\sigma$, the group $H^2(C_{p^r},F^*)$ is
generated by a cohomology class which admits a $2$-cocycle which
is determined by $u_{\sigma}^{p^r}=\zeta$ (see e.g.
\cite[p.31]{YamazakiUnschur}). Notice that this does not
necessarily hold without our assumption on the field. For example
$H^2(C_2,\mathbb{Q}^*)$ is an infinite group.

Let $G$ be a finite group and let $F$ be an algebraically closed
field of characteristic $0$. Then there exists a group $G^*$ with
an abelian normal subgroup $A\cong H^2(G,F^*)$ such that
\begin{equation*}
1 \rightarrow A \rightarrow G^* \rightarrow G \rightarrow 1
\end{equation*}
is a stem-extension, i.e. $A \leq Z(G^*) \cap (G^*)'$.
This $G^*$ is called a representation group of $G$ or a Schur
cover of $G$. Clearly, $|G^*|=|G||H^2(G,F^*)|$. In general the isomorphism type of a Schur cover is not unique,
but each cover satisfies
\begin{equation}\label{eq:GRSCHUR}
FG^* \cong \oplus_{[\alpha] \in H^2(G,F^*)} F^\alpha G.
\end{equation}
See \cite[Chapter 3, \S 3]{KarpilovskyProjective} for the details.

Different variations and generalizations of representation groups have been studied, see e.g. \cite{LassueurThevenaz, Sambonet} for some of the most recent.

The following example demonstrates that over non-algebraically closed fields there is no Schur cover, and at the same time suggests how to
find an analog, in a sense as in \eqref{eq:GRSCHUR}, in the non-algebraically closed case.
\begin{example}
Let $G= C_2$ be generated by an element $g$ and let $F=\mathbb{F}_5$. We
can define a $2$-cocycle $\beta \in Z^2(G,F^*)$ by $u_g^2=\zeta $
where $\zeta$ is of order $4$. Notice that $u_g$ is an element of
order $8$ in $F^\beta G$. It is clear that $H^2(G,F^*)\cong C_2$ and
therefore, if $G$ admits a Schur cover it is of order $4$. However,
$FC_4\cong F(C_2\times C_2) \cong 4F$ and in particular it does not
contain elements of order $8$. Consequently,~\eqref{eq:GRSCHUR} is
not satisfied and there is no Schur cover for $G$ over $F$.
However, it is not hard to check that
\begin{equation*}
\mathbb{F}_5C_8 \cong 4\mathbb{F}_5\oplus 2\mathbb{F}_{25}=2\left(\mathbb{F}_5 C_2\oplus \mathbb{F}_5^\beta
C_2\right)=2 \left(\oplus_{[\alpha] \in H^2(C_2,\mathbb{F}_5^*)} \mathbb{F}_5^\alpha
C_2\right).
\end{equation*}
\end{example}

We wish to find a group $G^*$ which will play a similar role of
the Schur cover over non-algebraically closed fields in the sense
that any twisted group ring over $G$ will be a direct summand of
the group ring over $G^*$. Since the construction of this group is
based on a proof of Yamazaki \cite{YamazakiUnschur} we will give
here the existence theorem with a sketch of the part of the proof
which describes how to construct this object. Again, for a field $F$ we
will denote by $t(F^*)$ the torsion part of $F^*$.

\begin{theorem}\label{Yamazaki cover}\cite{YamazakiUnschur} (see also \cite[Theorem 3.3.2]{KarpilovskyProjective})
Let $G$ be a finite group and let $F$ be a field such that $H^2(G,F^*)=H^2(G,t(F^*))$. There exists a finite central extension
\begin{equation}\label{eq:Yamazaki1}
1 \rightarrow A \rightarrow G^* \rightarrow G \rightarrow 1,
\end{equation}
such that any projective representation of $G$ is projectively equivalent to a linear representation
of $G^*$.
\end{theorem}

\textbf{Construction of $G^*$.}
First, we need to describe the group $A$ in~\eqref{eq:Yamazaki1}. Since $H^2(G,F^*)$ is a finite abelian group we may write
$$H^2(G,F^*)=\langle c_1 \rangle \times \langle c_2 \rangle \times \ldots \times \langle c_m \rangle.$$
Construct a new group as follows. Choose in any cohomology class $c_i$ a cocycle $\alpha _i$ of order $d_i$, let
$A_i = C_{d_i}$ and let
$$A=A_1\times A_2\times \ldots \times A_m.$$
Now, the group $G^*$ will be determined by a cohomology class $\beta \in H^2(G,A)$.
This $\beta$ can be considered as
$$(\beta _1, \beta _2,  \ldots , \beta _m) \in H^2(G,A_1)\times H^2(G,A_2) \times \ldots \times H^2(G,A_m),$$
while the only restriction on $\beta _i$ is that $\tilde{\chi}_i(\beta _i)=c_i$ for the natural morphism
$\tilde{\chi _i}:H^2(G,A_i)\rightarrow H^2(G,F^*).$ \qed

\begin{definition}
We will call the group $G^*$ in Theorem~\ref{Yamazaki cover} a \textit{Yamazaki cover} and will denote a Yamazaki cover of a group $G$ over a field $F$ by $Y_F(G)$.

If there is no proper quotient of $G^*$ which is also a
Yamazaki cover of $G$ we call $G^*$ a \textit{minimal Yamazaki
cover}.
\end{definition}

The following remarks are in order.

\begin{remark}
With the notations above we have a surjective morphism $\psi :A\rightarrow H^2(G,F^*)$. In fact this is the well-known transgression map $\operatorname{Hom}(A, F^*) \rightarrow H^2(G, F^*)$, cf. Definition~\ref{def:Tra} or \cite[Theorem 3.2.9]{KarpilovskyProjective}.
\end{remark}

\begin{remark}
Notice that with the above notations, $A$ is not uniquely determined, and in fact even its cardinality is not uniquely determined, since
there could be in $c_i$ cocycles $\alpha$ and $\alpha'$ of distinct orders.
Furthermore, like in the situation with the classical Schur cover, for a fixed $A$ different choices of $\beta$ can lead to non-isomorphic Yamazaki covers.
\end{remark}

\begin{remark} The existence of $Y_F(G)$ depends on the condition that $H^2(G, F^*) = H^2(G, t(F^*))$.
This condition was also investigated by Yamazaki. He showed that
$H^2(G, F^*) = H^2(G, t(F^*))$ if and only if $F^* =
(F^*)^{\operatorname{exp}(G/G')}t(F^*)$ \cite{YamazakiUnschur}
(cf. also \cite[Corollary 3.3.4]{KarpilovskyProjective}). In
particular over every finite field, the real and the complex
numbers Yamazaki covers always exist.
\end{remark}

The following is immediate now from Theorem~\ref{Yamazaki cover} and the construction of the Yamazaki cover.

\begin{corollary}
Let $Y_F(G)$ be a Yamazaki cover of a group $G$
over a field $F$ which corresponds to~\eqref{eq:Yamazaki1}. Then
\begin{equation*}
FY_F(G)\cong \frac{|A|}{|H^2(G,F^*)|} \oplus_{[\alpha] \in
H^2(G,F^*)} F^\alpha G.
\end{equation*}
\end{corollary}

For given groups $G$ and $H$ there is a well-known group theoretical condition how to determine whether $H$ is a Schur cover of $G$, assuming we know the order of $H^2(G, \mathbb{C}^*)$ \cite[Theorem 3.3.7]{KarpilovskyProjective}. For minimal Yamazaki covers we can provide a similar criterion  which requires a few more things to check though. For an abelian group $A$ and a prime $p$ denote by $A_p$ the Sylow $p$-subgroup of $A$.

\begin{theorem}\label{th:YamazakiGT}
Let $1 \rightarrow Z \rightarrow H \rightarrow G \rightarrow 1$ be
a central extension of a finite group $G$ and $F$ a field such that $H^2(G,F^*)\cong H^2(G,t(F^*))$. Assume that
this extension satisfies the following:
\begin{itemize}
\item $Z \cap H' \cong \operatorname{Hom}(M(G), F^*)$.
\item $\operatorname{rk}(G/G') = \operatorname{rk}(H/H')$.
\item For each prime $p$ we have the following: If $F^*$ contains a maximal finite $p$-subgroup and the order of this group is $p^m$ then $(Z/Z\cap H')_p$ is a direct product of $\operatorname{rk}(\operatorname{Ext}((G/G')_p, F^*))$ cyclic $p$-groups of order $p^m$.
\item $H' \cap Z$ has a complement in $Z$, i.e. the short exact sequence $1 \rightarrow Z \cap H' \rightarrow Z \rightarrow Z/(Z \cap H') \rightarrow 1$ is split.
\end{itemize}
Then $H$ is a minimal Yamazaki cover of $G$ over $F$.
\end{theorem}

\begin{proof}
A diagram illustrating the steps of the proof can be found below in \eqref{eq:picture}.

Note that by assumption the exponent of $Z$ divides the exponent of $F^*$, so
$Z \cong \operatorname{Hom}(Z, F^*)$. We need to show that the
transgression map (see Definition~\ref{def:Tra}) $\Tra:
\operatorname{Hom}(Z, F^*) \rightarrow H^2(G, F^*)$ is surjective
and moreover that this is not the case for any central extension $1
\rightarrow Z/\tilde{Z} \rightarrow H/\tilde{Z} \rightarrow G
\rightarrow 1$ for $\tilde{Z}$ a proper subgroup of $Z$.

Let $Z = (Z \cap H') \times C$ for a subgroup $C$ of $Z$ and
identify $C$ and $Z/(Z \cap H')$. By our assumption that $Z \cap
H' \cong \operatorname{Hom}(M(G), F^*)$ and \cite[Lemma
11.5.1]{KarpilovskyVolII} it follows that the image of
$\operatorname{Tra}|_{Z \cap H'}$ is isomorphic to
$\operatorname{Hom}(M(G), F^*)$. Define $H^2_0(G, F^*)$ as in
\cite[Definition before Theorem 2.2.9]{KarpilovskyProjective} to
be the part of $H^2(G,F^*)$ which corresponds to all central
extensions $1 \rightarrow A \rightarrow E \rightarrow G
\rightarrow 1$ with the property that $A' \cap E = 1$. Then
\cite[Theorem 2.2.9]{KarpilovskyProjective} implies that $H^2_0(G,
F^*)$ is exactly the image of $\operatorname{Ext}(G/G', F^*)$
under the inflation  map. In particular the transgression map
$\Tra: C \rightarrow H^2(G, \mathbb{F}^*)$ related to the short
exact sequence
$$1 \rightarrow Z/(Z \cap H') \rightarrow H/(Z\cap H') \rightarrow G \rightarrow 1 $$
has an image lying in $H^2_0(G, F^*)$. It remains to show that this is indeed the whole image and that this is not the case for any group smaller than $H/(Z \cap H')$. It is enough to show this for a non-trivial Sylow $p$-subgroup $P$ of $C$ for some fixed prime $p$ with respect to the Sylow $p$-subgroup of $H^2_0(G, (F^*)_p)$ as it follows for each Sylow subgroup of $C$ in the same way.

It follows from our second and third assumptions that $\rk((H/H')_p) = \rk((G/G')_p)$.
Let $P = \langle a_1 \rangle \times \langle a_2 \rangle \times ... \times \langle a_r \rangle$ for some $a_1$,...,$a_r$. Then each $a_i$ has order $p^m$ by assumption and $r = \rk(\operatorname{Ext}((G/G')_p, (F^*)_p)) = \rk(H^2_0(G, (F^*)_p))$. Fix some $1 \leq i \leq r$. An abelian extension of $G/G'$ by $(F^*)_p$ corresponding to $a_i$ is not of the form $1 \rightarrow (F^*)_p \rightarrow (F^*)_p \times G/G' \rightarrow G/G' \rightarrow 1$, as $\rk((H/H')_p) = \rk((G/G')_p)$. So by \cite[Theorem 2.1.2 and Corollary 2.1.3]{KarpilovskyProjective} the coclass $\Tra(a_i)$ is not a coboundary for any $1 \leq i \leq r$. So $\rk(\Tra(P)) = \rk(H^2_0(G, (F^*)_p))$.

Assume that $\Tra(P)$ is a proper subgroup of $H^2_0(G, (F^*)_p)$. Then there is an $1\leq i \leq r$ and a cocycle $b \in Z^2(G, (F^*)_p)$ such that $b^p = \operatorname{Tra}(a_i)$. But then the $b$ must have a value which is a $p^{m+1}$-th primitive root of unity in $F^*$, contradicting our choice of $m$.

Lastly, the minimality of $H$, follows from the fact that $a_i$
corresponds to an element in $\operatorname{Ext}(G/G', (F_p)^*)$,
that is an abelian extension with kernel $C_{p^m}$ and hence $a_i$
must have order at least $p^m$.

\begin{center}
\begin{equation}\label{eq:picture}
\begin{tikzpicture}[baseline=(current  bounding  box.center)]
\tikzset{thick arc/.style={->, black, fill=none,  >=stealth,
text=black}} \tikzset{node distance=2cm, auto}
 \node (triv_UPLEFT){$1$};
 \node (Z_UP) [right of=triv_UPLEFT] {$Z$};
  \tikzset{node distance=3cm, auto}
 \node (H_UP) [right of=Z_UP] {$H$};
  \node (G_UP) [right of=H_UP] {$G$};
   \tikzset{node distance=2cm, auto}
  \node (triv_UPRIGHT)[right of=G_UP]{$1$};

 \tikzset{node distance=2cm, auto}
   \node (triv_MIDDLELEFT)[below of=triv_UPLEFT] {$1$};
  \node (Zmod_MIDDLE) [right of=triv_MIDDLELEFT] {$Z/(Z\cap H')$};
   \tikzset{node distance=3cm, auto}
 \node (Hmod_MIDDLE) [right of=Zmod_MIDDLE] {$H/(Z\cap H')$};
  \node (G_MIDDLE) [right of=Hmod_MIDDLE] {$G$};
   \tikzset{node distance=2cm, auto}
  \node (triv_MIDDLERIGHT)[right of=G_MIDDLE]{$1$};

 \tikzset{node distance=2cm, auto}
   \node (triv_DOWNLEFT)[below of=triv_MIDDLELEFT] {$1$};
  \node (Zmod_DOWN) [right of=triv_DOWNLEFT] {$Z/(Z\cap H')$};
   \tikzset{node distance=3cm, auto}
 \node (Hmod_DOWN) [right of=Zmod_DOWN] {$H/H'$};
  \node (Gmod_DOWN) [right of=Hmod_DOWN] {$G/G'$};
   \tikzset{node distance=2cm, auto}
  \node (triv_DOWNRIGHT)[right of=Gmod_DOWN]{$1$};

\draw[thick arc, draw=black] (triv_UPLEFT) to node [above] {$$} (Z_UP);
\draw[thick arc, draw=black] (Z_UP) to node [above] {$$} (H_UP);
\draw[thick arc, draw=black] (H_UP) to node [above] {$$} (G_UP);
\draw[thick arc, draw=black] (G_UP) to node [above] {$$} (triv_UPRIGHT);

\draw[thick arc, draw=black] (triv_MIDDLELEFT) to node [above] {$$} (Zmod_MIDDLE);
\draw[thick arc, draw=black] (Zmod_MIDDLE) to node [above] {$$} (Hmod_MIDDLE);
\draw[thick arc, draw=black] (Hmod_MIDDLE) to node [above] {$$} (G_MIDDLE);
\draw[thick arc, draw=black] (G_MIDDLE) to node [above] {$$} (triv_MIDDLERIGHT);

\draw[thick arc, draw=black] (triv_DOWNLEFT) to node [above] {$$} (Zmod_DOWN);
\draw[thick arc, draw=black] (Zmod_DOWN) to node [above] {$$} (Hmod_DOWN);
\draw[thick arc, draw=black] (Hmod_DOWN) to node [above] {$$} (Gmod_DOWN);
\draw[thick arc, draw=black] (Gmod_DOWN) to node [above] {$$} (triv_DOWNRIGHT);

\draw[thick arc, draw=black] (Z_UP) to node [left] {$ $} (Zmod_MIDDLE);
\draw[thick arc, draw=black] (Zmod_MIDDLE) to node [left] {$ $} (Zmod_DOWN);
\draw[thick arc, draw=black] (H_UP) to node [left] {$ $} (Hmod_MIDDLE);
\draw[thick arc, draw=black] (Hmod_MIDDLE) to node [left] {$ $} (Hmod_DOWN);
\draw[thick arc, draw=black] (G_UP) to node [left] {$ $} (G_MIDDLE);
\draw[thick arc, draw=black] (G_MIDDLE) to node [left] {$ $} (Gmod_DOWN);

\end{tikzpicture}
\end{equation}
\end{center}

\end{proof}

\begin{example} We provide an example for $Y_{\mathbb{F}_3}(D_8)$ where $D_8$ denotes a dihedral group of order $8$. We also give an example that the last condition in Theorem~\ref{th:YamazakiGT} is necessary. Let $G = D_8$ and $F = \mathbb{F}_3$.

We have $G/G' \cong C_2 \times C_2$, so $\operatorname{Ext}(G/G', F^*) \cong C_2 \times C_2$. Moreover $M(G) \cong C_2$ \cite[Proposition 4.6.4]{KarpilovskyProjective}. A minimal Yamazaki cover of $G$ is given by
\begin{align*}
Y(G) = (\langle a \rangle \times \langle b \rangle ) \rtimes \langle c \rangle: \ \ a^2 = 1, \ b^8 = 1, \ c^4 = 1, \ a^c = a, \ b^c = ab^3.
\end{align*}
Then $Z(Y(G)) = \langle a, b^4, c^2 \rangle$ is an elementary abelian group of order $8$. Moreover $Y(G)' = \langle ab^2 \rangle$ is a cyclic group of order $4$. Setting $Z = Z(Y(G))$ we observe that all conditions from Theorem~\ref{th:YamazakiGT} are satisfied. Using the package \texttt{Wedderga} \cite{Wedderga} of the computer algebra system \texttt{GAP} \cite{GAP} we obtain moreover
$$FY(G) \cong 4\mathbb{F}_3 \oplus 6\mathbb{F}_9 \oplus 8M_2(\mathbb{F}_3) \oplus 2M_2(\mathbb{F}_9).$$

We now exhibit an example that the last condition in Theorem~\ref{th:YamazakiGT} is necessary. Set
\begin{align*}
H = (\langle a \rangle \times \langle b \rangle ) \rtimes \langle c \rangle: \ \ a^4 = 1, \ b^4 = 1, \ c^4 = 1, \ a^c = a^{-1}, \ b^c = ab.
\end{align*}
Then we have that $Z(H) = \langle ab^2, c^2 \rangle \cong C_4 \times C_2$ and $H' = \langle a \rangle \cong C_4$. Set $Z = Z(H)$. Then $H/Z \cong G$, $Z \cap H' = \langle a^2 \rangle \cong C_2$, $\operatorname{rk}(Z/(Z \cap H')) = 2$, $Z/(Z \cap H') \cong C_2 \times C_2$ and $\operatorname{rk}(H/H') = 2$. So the extension $1 \rightarrow Z \rightarrow H \rightarrow G \rightarrow 1$ satisfies all the conditions of Theorem~\ref{th:YamazakiGT} except the last one. But $H$ is not a Yamazaki cover of $G$ as its group algebra over $F$ is not isomorphic with the group algebra of $Y(G)$ given above. Indeed,
$$FH \cong 4\mathbb{F}_3 \oplus 6\mathbb{F}_9 \oplus 4M_2(\mathbb{F}_3) \oplus 4M_2(\mathbb{F}_9),$$
which again can be calculated using \cite{Wedderga}.
\end{example}

\section{The Dade example}\label{Dade}
In 1971 E. Dade, answering a question of R. Brauer \cite[Problem
2*]{Brauer63}, provided a family of examples of non-isomorphic
finite groups $G$ and $H$ such that the group algebras of $G$ and
$H$ are isomorphic over any field $F$. We will show that for a
subclass of Dade's examples there are fields $F$ such that $G
\not\sim_F H$. Note that the groups of Dade are metabelian and
hence have non-isomorphic group rings over the integers, a result
due to Whitcomb already known at the time Dade solved Brauer's problem \cite{Whitcomb}.

We will first describe the groups given by Dade. Let $p$ and $q$
be primes such that $q \equiv 1 \mod p^2$ and let $w$ be an
integer such that $w \not \equiv 1 \mod q^2$, but $w^p \equiv 1
\mod q^2$. Let $Q_1$ and $Q_2$ be the following two non-abelian
groups of order $q^3$.
\begin{align*}
Q_1 &= (\langle \tau_1 \rangle \times \langle \sigma_1 \rangle) \rtimes \langle \rho_1 \rangle, \\
Q_2 &= \langle \sigma_2 \rangle \rtimes \langle \rho_2 \rangle, \\
\tau_1^q &= \sigma_1^q = \rho_1^q = \sigma_2^{q^2} = \rho_2^q = 1, \ \sigma_2^q =: \tau_2, \\
\tau_1^{\rho_1} &= \tau_1, \ \sigma_1^{\rho_1} = \tau_1\sigma_1, \ \sigma_2^{\rho_2} = \tau_2\sigma_2
\end{align*}
So $Q_1$ and $Q_2$ are just the two non-abelian groups of order $q^3$ such that $Q_1$ has exponent $q$ (aka the Heisenberg group).

Let $\langle \pi_1 \rangle \cong C_{p^2}$, $\langle \pi_2 \rangle \cong C_p$ and for $i,j \in \{1,2 \}$ let
$$\rho_i^{\pi_j} = \rho_i, \ \sigma_i^{\pi_j} = \sigma_j^w, \ \tau_i^{\pi_j} = \tau_i^w.$$
Define two groups by
\begin{align*}
G &= (Q_1 \rtimes \langle \pi_1 \rangle) \times (Q_2 \rtimes \langle \pi_2 \rangle), \\
H &= (Q_1 \rtimes \langle \pi_2 \rangle) \times (Q_2 \rtimes \langle \pi_1 \rangle).
\end{align*}
These are the groups constructed by Dade as a counterexample to Brauer's question.

Notice that $G=G_1\times G_2$ and $H=H_1\times H_2$ for
$$G_1=Q_1 \rtimes \langle \pi_1 \rangle,\quad G_2=Q_2 \rtimes \langle \pi_2 \rangle,\quad
 H_1=Q_1 \rtimes \langle \pi_2 \rangle \quad H_2=Q_2 \rtimes \langle \pi_1 \rangle.$$

 \subsection{The second cohomology groups of $G$ and $H$}
In order to calculate the Schur multipliers of $G$ and $H$ we will
use a result of Schur \cite{Schur} about the Schur multiplier of
direct products of groups (see also \cite[Corollary 2.3.14]{KarpilovskyProjective}). Define the tensor product of two
finite groups $A$ and $B$ by
\begin{equation*}
  A\otimes B=A/A' \otimes _{\mathbb{Z}} B/B'.
\end{equation*}

\begin{theorem}\label{th:directproductSchur}
 Let $A$ and $B$ be finite groups. Then
\begin{equation*}
  M(A \times B)=M(A)\times M(B)\times (A \otimes B).
\end{equation*}
\end{theorem}

Notice that (slightly abusing notation)
$$G_1'=\langle \tau_1 \rangle \times \langle \sigma_1 \rangle=H_1' \text{ and }
G_2'= \langle \sigma_2 \rangle=H_2'.$$
Therefore
\begin{equation*}
G_1/G_1'\cong C_q\times C_p\cong H_2/H_2' \text{ and } G_2/G_2'\cong C_q\times C_{p^2}\cong H_1/H_1'.
\end{equation*}
Consequently
\begin{equation*}
G_1\otimes G_2\cong H_1\otimes H_2\cong C_q\times C_p.
\end{equation*}
We will use Theorem~\ref{th:directproductSchur} to compute the Schur multipliers of $G_1$, $G_2$, $H_1$ and $H_2$.
Notice that $G_1$, $G_2$, $H_1$ and $H_2$ are written as semi-direct products of subgroups of coprime order. The following lemma will be of use.

\begin{lemma}(see \cite[Corollary 2.2.6]{KarpilovskySchur})\label{lemma:semicoprimeSchur}
  Let $N$ and $T$ be subgroups of a group $G$ of co-prime order and assume $G = N \rtimes T$. Then
  \begin{equation*}
    M(G)=M(T)\times M(N)^T.
  \end{equation*}
Here $M(N)^T$ are the elements in $M(N)$ which are invariants under the $T$-action.
\end{lemma}

First, by \cite[Theorem 4.7.3]{KarpilovskyProjective},
$M(Q_1)\cong C_q\times C_q$ and $Q_2$ admits a trivial Schur multiplier.
Therefore, since a Schur multiplier of a cyclic group is trivial, we get by Lemma~\ref{lemma:semicoprimeSchur} that
\begin{equation*}
M(G_2)=M(H_2)=1.
\end{equation*}
We are left with the computation of $M(G_1)$ and $M(H_1)$.
As written above $M(Q_1)\cong C_q\times C_q$. In fact, $M(Q_1)$ is generated by the cohomology classes $\alpha$ and $\beta$ which are determined by the
following relations in the corresponding twisted group algebras $\mathbb{C}^{\alpha}Q_1$ and $\mathbb{C}^{\beta}Q_1$
\begin{align*}
\alpha:& [u_{\tau},u_{\sigma}]=\zeta,\quad [u_{\tau},u_{\rho}]=1, \\
\beta:& [u_{\tau},u_{\sigma}]=1,\quad [u_{\tau},u_{\rho}]=\zeta.
\end{align*}
Here $\zeta$ denotes a primitive $q$-th roots of unity.
Notice, that for $i=1,2$,
$$[u_{\tau^{\pi_i}},u_{\rho^{\pi_i}}]=[u_{\tau^w},u_{\rho}]=[u_{\tau},u_{\rho}]^w.$$
Therefore, $\beta$ is not invariant under the action of $\langle \pi _i \rangle$ for $i=1,2$.
We need to check whether $\alpha$ is invariant.
It turns out that $\alpha$ is invariant if and only if $p=2$.
Indeed, in $\mathbb{C}^{\alpha}Q_1$
$$[u_{\tau^{\pi_i}},u_{\sigma^{\pi_i}}]=[u_{\tau}^w,u_{\sigma}^w]=\zeta^{w^2}.$$
Therefore, $\alpha$ is invariant if and only if $w^2 \equiv 1 \mod q$ which happens if and only if $p=2$ because $w^p \equiv 1 \mod q^2$.
As a consequence of the above we obtain the following.

\begin{proposition}\label{prop_SchurMultiplier}
With the above notations, if $p=2$
$$M(G) \cong M(H)  \cong C_q\times C_q\times C_p,$$
and for $p>2$
$$M(G) \cong M(H) \cong C_q\times C_p.$$
\end{proposition}

We proceed to construct $H^2(G,F^*)$ using the exact sequence given in \eqref{eq:UCT}.
Observe that
$$G/G'\cong C_q\times C_q\times C_2\times C_4=\langle G'\rho_1 \rangle \times \langle G'\rho_2 \rangle \times \langle G'\pi_1 \rangle \times \langle G'\pi_2 \rangle.$$

Therefore, by equations~\eqref{eq:EXTdecomposition} and~\eqref{eq:cohoofcyclic} we get
$$\operatorname{Ext}(G/G',F^*)\cong C_q\times C_q\times C_p\times C_{p^2}.$$
\begin{corollary}\label{cor:H2Dade}
For $p = 2$ we have
\begin{equation*}
H^2(G,F^*)\cong \operatorname{Ext}(G/G',F^*)\times
\operatorname{Hom}(M(G),F^*)\cong \left(C_q\times C_q\times C_p\times C_{p^2}
\right)\times \left(C_q\times C_q\times C_p \right).
\end{equation*}
and for $p >2$ we get
\begin{equation*}
H^2(G,F^*)\cong \operatorname{Ext}(G/G',F^*)\times
\operatorname{Hom}(M(G),F^*)\cong \left(C_q\times C_q\times C_p \times C_{p^2}
\right)\times \left(C_q\times C_p \right).
\end{equation*}
\end{corollary}
Notice that all the arguments above about $G$ are true also for $H$.

\subsection{The Yamazaki covers of $G$ and $H$}
From now on we will assume that $p=2$ and $q$ is any prime
satisfying the relations in Dade's groups. Note that we can then assume w.l.o.g. $w=-1$. Moreover we assume that $F=\mathbb{F}_r$ is a finite
field such that
\begin{itemize}
\item $r-1$ is divisible by $q$ but not by $q^2$,
\item $r-1$ is divisible by $4$ but not by $8$ and
\item $r^2-1$ is divisible by $8$ but not by $16$.
\end{itemize}
There exist infinitely many such fields, e.g. by Dirichlet's theorem on primes in arithmetic progressions.

This allows us to give the Yamazaki covers of $G$ and $H$ using less notation, though it is not hard to give them also in case $p>2$. But the difference observed between the Schur multipliers in Proposition~\ref{prop_SchurMultiplier} turns out to be crucial for our arguments, so we concentrate on this case. See Remark~\ref{rem:OddCase} about the case $p>2$.

Let $\zeta$ be a primitive $q$-th and $\xi$ a primitive $4$-th root of unity in $F$. In order to construct the Yamazaki covers of $G$ and $H$ we will need to describe the group $A$ in the construction after Theorem~\ref{Yamazaki cover} as computed in the previous subsection and in particular in Corollary~\ref{cor:H2Dade}.
Let
\begin{align*}
H^2(G,F^*) =&  \operatorname{Hom}(M(G),F^*) \times \operatorname{Ext}(G/G',F^*) \\
 =& (\langle  \alpha \rangle \times \langle \beta \rangle \times \langle \gamma \rangle) \times (\langle \kappa \rangle \times \langle \lambda \rangle \times \langle \mu \rangle \times \langle \nu \rangle),
\end{align*}
where
\begin{itemize}
\item $\alpha$ is of order $q$, determined by
$[u_{\rho_1},u_{\rho_2}]=\zeta$.
 \item $\beta$ is of order $2$,
determined by $[u_{\pi_1},u_{\pi_2}]=-1$.
 \item $\gamma$ is of
order $q$, determined by $[u_{\rho_1},u_{\sigma_1}]=\zeta$. \item
$\kappa$ is of order $q$ determined by $u_{\rho_1}^q=\zeta$. \item
$\lambda$ is of order $q$ determined by $u_{\rho_2}^q=\zeta$.
\item $\mu$ is of order $4$ determined by $u_{\pi _1}^4=\xi$.
\item $\nu$ is of order $2$ determined by $u_{\pi _2}^4=\xi$.
\end{itemize}

Notice, that from the above the only cohomology class in which the order of the cocycle is bigger than the order of the cohomology class is
for $\nu$. Here the order of $\nu$ is $2$ and the order of the corresponding cocycle is $4$. Therefore we may consider the extending group $A$ to be like $H^2(G,F^*)$ with the only difference being that the $C_2$ generated by $\nu$ in $H^2(G,F^*)$ will have a representative cocycle $\bar{\nu}$ in $A$ which will generate a $C_4$.

Now in order to construct the Yamazaki cover we need to construct a cohomology class $\beta _{(G,A)}\in H^2(G,A)$ which will correspond to
the central extension~\eqref{eq:Yamazaki1}. Let $\{\tilde{g}\}_{g \in G}$ be a section of $G$ in $G^*$ corresponding to ~\eqref{eq:Yamazaki1}.  Then, abusing notation, $\beta _{(G,A)}$ can be chosen to be the cohomology class determined by (compare with the classes given above)
\begin{align*}
&[u_{\rho_1}, u_{\rho_2}]=\zeta ,\quad [u_{\pi_1}, u_{\pi_2}]=-1 ,\quad [u_{\rho_1}, u_{\sigma _1}]=\zeta , \\
&u_{\rho_1}^q=\zeta, \quad u_{\rho_2}^q=\zeta , \quad u_{\pi _1}^4=\xi ,\quad u_{\pi _2}^2=\xi.
\end{align*}

This leads us also to the Yamazaki covers of $G$ and $H$ over $F$. Since from now on we will only work with these covers and their subgroups we will use the same notations for the elements as before in the ``uncovered'' groups. Here we will introduce cyclic subgroup $\langle x \rangle$, $\langle y \rangle$ and $\langle z \rangle$ corresponding to the cohomology classes $\alpha$, $\beta$ and $\gamma$ respectively. The orders of the other generators change according to the cohomology classes $\kappa$, $\lambda$, $\mu$ and $\nu$.
We will construct both Yamazaki covers as the quotient of the same infinite group.

\textbf{Notation:} Let $Y$ be a group generated by elements $\sigma_1$, $\sigma_2$, $\rho_1$, $\rho_2$, $\tau_1$, $\pi_1$, $\pi_2$, $x$, $y$ and $z$ subject to the following relations:

\begin{align*}
 \ \sigma_2^{q^2} &= x^q = \rho_2^{q^2} = y^2 = \pi_2^8 = z^q = \tau_1^q = \rho_2^{q^2} = \sigma_1^q = \pi_1^{16} = 1, \ \sigma_2^q =: \tau_2, \\
\sigma_2^{\rho_2} &= \sigma_2 \tau_2, \  \rho_1^{\sigma_1} = \tau_1 \rho_1, \ \tau_1^{\sigma_1} = z\tau_1, \ \rho_2^{\rho_1} = x\rho_2. \\
\end{align*}

Moreover we have $x, y, z \in Z(Y)$ and unless otherwise specified in the relations above for $g,h \in \{ \sigma_1, \sigma_2, \rho_1, \rho_2, \tau_1\}  $ we have $[g,h] = 1$ in $Y$.

\begin{lemma}\label{lemma_YamazakiCover}
Let $Y$ be the group described above. Let $Y(G)$ be the quotient of $Y$ in which $\pi_i$ commutes with $\sigma_j$, $\rho_j$, $\tau_j$ for $i \neq j$ and which is additionally subject to the following relations
$$\sigma_1^{\pi_1} = \sigma_1^{-1}, \ \sigma_2^{\pi_2} = \sigma_2^{-1}, \ \tau_1^{\pi_1} = z\tau_1^{-1}, \ \ \pi_2^{\pi_1} = y\pi_2. $$
Let $Y(H)$ be the quotient of $Y$ in which $\pi_i$ commutes with $\sigma_i$, $\rho_i$, $\tau_i$ for $i \in \{1,2\}$ and where additionally we have the relations
$$\sigma_2^{\pi_1} = \sigma_2^{-1}, \ \tau_1^{\pi_2} = z\tau_1^{-1}, \ \sigma_1^{\pi_2} = \sigma_1^{-1}, \ \pi_1^{\pi_2} = y\pi_1.$$

Then $Y(G)$ and $Y(H)$ are minimal Yamazaki covers of $G$ and $H$ respectively.
\end{lemma}

\textbf{Remark:} Using semi-direct products one can write:
\begin{align*}
Y(G) &= (( \langle \sigma_2 \rangle \rtimes (\langle x \rangle \times \langle \rho_2 \rangle )) \rtimes (\langle y \rangle \times \langle \pi_2 \rangle )) \rtimes ((( \langle z \rangle \times \langle \tau_1 \rangle \times \langle \rho_1 \rangle) \rtimes \langle \sigma_1 \rangle ) \rtimes \langle \pi_1 \rangle), \\
Y(H) &=  (( \langle \sigma_2 \rangle \rtimes (\langle x \rangle \times \langle \rho_2 \rangle )) \rtimes (\langle y \rangle \times \langle \pi_1 \rangle )) \rtimes ((( \langle z \rangle \times \langle \tau_1 \rangle \times \langle \rho_1 \rangle) \rtimes \langle \sigma_1 \rangle ) \rtimes \langle \pi_2 \rangle).
\end{align*}
Note that the only difference when writing this way is an interchange between $\pi_1$ and $\pi_2$.

\begin{proof}
We will use Theorem~\ref{th:YamazakiGT} and Corollary~\ref{cor:H2Dade}. In the notation of Theorem~\ref{th:YamazakiGT} we have
$$Z = \langle x \rangle \times \langle y \rangle \times \langle z \rangle \times \langle \rho_1^q \rangle \times \langle \rho_2^q \rangle  \times \langle \pi_1^4 \rangle \times \langle \pi_2^4 \rangle. $$
Moreover
$$Y(G)' = \langle x,y,z,\sigma_2,\tau_1,\sigma_1 \rangle.$$
So $Y(G)' \cap Z = \langle x \rangle \times \langle y \rangle \times \langle z \rangle \cong \text{Hom}(M(G), F^*)$. The other conditions are now easy to check.

The same statements hold for $Y(H)$, even using formally the same elements.
\end{proof}

\subsection{Proof of Theorem 1}
We keep the assumptions from the previous subsection and we will show that in this case $G \not \sim_F H$. We will use the minimal Yamazaki covers $Y(G)$ and $Y(H)$ introduced in Lemma~\ref{lemma_YamazakiCover} and explicit elements will refer to these groups.

To show that $G$ and $H$ are not in relation over $F$ we will work
with Wedderburn decompositions of $FY(G)$ and $FY(H)$.
The groups $Y(G)$ and $Y(H)$ are supersolvable as can be seen by their defining relations and hence both groups are monomial, i.e. each irreducible character of these groups is induced by a linear character of a subgroup. This holds over $\mathbb{C}$ by \cite[Theorem 6.22]{Isaacs} and over finite fields of characteristic not dividing $|G|$ by \cite[Corollary 8]{BrochedelRio}.

Each Wedderburn component of $FY(G)$ and $FY(H)$ corresponds to
a Wedderburn component of a twisted group algebra $F^\varphi G$
and $F^\varphi H$ respectively. Let $B$ be such a Wedderburn
component. Then in fact we can easily determine $\varphi$ from the
character $\chi$ corresponding to $B$. Namely if we view $\varphi$
as a product of powers of the generators $\alpha$, $\gamma$, $\kappa$, $\lambda$, $\beta$, $\mu$ and $\nu$, then
we can read of $\varphi$ from the powers of $\zeta$, $-1$ and
$\xi$ appearing in the values of $\varphi$ on $x$, $z$, $\rho_1^q$, $\rho_2^q$, $y$,
$\pi_1^{4}$ and $\pi_2^{4}$ respectively.
This follows from the natural correspondence between projective
representations and 2-cocycles as explained in Section~\ref{sec:ProjRep} .

Denote by $F_2$ the field obtained from adjoining a primitive $8$-th root of unity to $F$ and by $F_4$ the field obtained
from adjoining a primitive $16$-th root of unity to $F$. Note that these fields are different by our choice of $F$.

We will show that there is a cohomology class $\psi$ in
$H^2(G,F^*)$ such that every Wedderburn component of $F^\psi G$ is
a matrix ring over the field $F_4$, but there is no cohomology
class $\varphi$ in $H^2(H, F^*)$ such that $F^\varphi H$ is the
direct sum of matrix rings over $F_4$.
  This will be proven in the next two lemmas and clearly imply $G \not \sim_F H$.

\begin{lemma}
For $\psi = \gamma \mu$ the Wedderburn decomposition of $F^\psi G$ is a direct sum of matrix rings over $F_4$.
\end{lemma}

\begin{proof}
Both $\gamma$ and $\mu$ only influence the subgroup
$G_1 = Q_1 \rtimes \langle \pi_1 \rangle$, in the sense that we can choose a cocycle $\psi'$ representing $\psi$ such that $\psi'((g_1,g_2),(1,\tilde{g}_2)) = 1$ for every $g_1 \in G_1$ and $g_2, \tilde{g}_2 \in G_2$.
So $k^\psi G = kG_2 \otimes k^{\psi_1}G_1$, where $\psi_1$ denotes the restriction of
$\psi$ to $G_1$. It is hence sufficient to show that
$k^{\psi_1}G_1$ is a direct sum of matrix rings over $F_4$.
A minimal Yamazaki cover of $G_1$ over $F$ is given by
$$Y(G_1) = (( \langle z \rangle \times \langle \tau_1 \rangle \times \langle \rho_1 \rangle) \rtimes \langle \sigma_1 \rangle ) \rtimes \langle \pi_1 \rangle $$
where the orders of the generators and the relations between them are exactly as in $Y(G)$.

The Wedderburn decompositions of $FY(G_1)$ can also be computed in positive characteristic as described in \cite{BrochedelRio}. In particular each Wedderburn component corresponds to a pair $(S,T)$ of subgroups in $Y(G_1)$ such that $S$ has a linear character $\chi$ with kernel $T$ and the induction $\operatorname{ind}_S^{Y(G_1)}(\chi)$ of $\chi$ to $Y(G_1)$ is irreducible.
Moreover assume that $\operatorname{ind}_S^{Y(G_1)}(\chi)$ corresponds to some Wedderburn component of $F^{\psi_1}G_1$, i.e. we have $z, \pi_1^{8} \notin T$ and $\rho_1^q \in T$. Our claim will follow once we show that $S$ necessarily contains an element of order $16$ or equivalently:

\vspace*{.2cm}
\textit{Claim:} Every irreducible character of $Y(G_1)$ whose kernel contains $\rho_1$, but not $z$ and $\pi_1^8$, has odd degree.
\vspace*{.2cm}

The claim is true over $F$ if and only if it is true over $\mathbb{C}$.
To make the calculations a bit easier we use the bar-notation to denote the natural projection modulo $\langle \rho_1^q, \pi_1^8 \rangle$ and the reduction of $Y(G_1)$ and set $R = Y(G_1)/\langle \rho_1^q, \pi_1^8 \rangle$. We will
prove that any irreducible character of
$R$ whose kernel does
not contain $z$ has odd degree which will imply the claim.

First of all observe that $\langle \bar{z}, \bar{\tau}_1,
\bar{\rho}_1 \rangle$ is an abelian normal subgroup of $R$ of
index $2q$ and so Ito's Theorem \cite[Theorem 6.15]{Isaacs}
implies that the character degree of each irreducible character of
$R$ divides $2q$. So each irreducible character of odd degree has
degree $1$ or $q$. Note that the number of characters of degree
$1$ of $R$ equals $|R/R'| = |R/\langle \bar{z}, \bar{\tau}_1,
\bar{\sigma}_1 \rangle| = 2q$. By \cite[Theorem 13.26]{Isaacs}, a very special version of the McKay-conjecture, the
number of irreducible characters of odd degree of $R$ is the same
as that of $N_R(\langle \bar{\pi}_1 \rangle)$, the normalizer in $R$ of the cyclic subgroup $\langle \bar{\pi}_1 \rangle$. Now $N_R(\langle \bar{\pi}_1 \rangle) = \langle
\bar{z}, \bar{\rho}_1, \bar{\pi}_1 \rangle$ is an abelian group of
order $2q^2$ and has $2q^2$ irreducible characters of odd degree.
Moreover $R/\langle \bar{z} \rangle$ has also $2q$ characters of
degree $1$ and $\frac{q(q-1)}{2}$ irreducible characters of degree
$2$ which are those having $\bar{\tau}_1$ in its kernel. This follows since
$$R/\langle \bar{z}, \bar{\tau}_1 \rangle \cong \langle \bar{\rho}_1 \rangle \times (\langle \bar{\sigma}_1 \rangle \rtimes \langle \bar{\pi}_1 \rangle) \cong C_q \times D_{2q},$$
where $D_{2q}$ denotes a dihedral group of order $2q$, and $D_{2q}$ has exactly $\frac{q-1}{2}$ irreducible characters of
degree $2$. Moreover the subgroup $\langle \bar{\tau}_1,
\bar{\rho}_1, \bar{\sigma}_1 \rangle$ of $R/ \langle \bar{z}
\rangle$, which is an extraspecial $q$-group, has $q-1$
irreducible characters of degree $q$, see e.g. \cite[Theorem
31.5]{DornhoffA}. The induction of each of these characters, which
are all not real-valued, to $R/ \langle \bar{z} \rangle$ is
irreducible, since it is real on the real conjugacy class of
$\bar{\tau}_1$, and two of them induce the same character. So $R/\langle
\bar{z} \rangle$ has $\frac{q-1}{2}$ irreducible characters of
degree $2q$. Summing the squares of the degrees of the irreducible
characters of $R/\langle \bar{z} \rangle$ obtained so far we obtain
$$2q\cdot1^2 + \frac{q(q-1)}{2}\cdot2^2 + \frac{q-1}{2}\cdot(2q)^2 = 2q^3.$$
So there are no further irreducible characters of $R/\langle \bar{z} \rangle$. In particular from
all irreducible odd degree characters of $R$ only the $2q$ linear
characters of $R$ have $\bar{z}$ in its kernel. But since any other
irreducible odd degree character has degree $q$, there are $2q^2$
such characters and since
$$(2q^2-2q)q^2 = 2q^4 - 2q^3 = |R| - |R/ \langle \bar{z} \rangle|$$
these are actually all irreducible characters of $R$ which do not contain $\bar{z}$ in its center. Hence the claim follows. This also finishes the proof of the lemma.
\end{proof}

\begin{lemma}
There is no $\varphi \in H^2(H,F^*)$
such that every direct summand of $F^\varphi H$ is a matrix
algebra over $F_4$.
\end{lemma}

\begin{proof}
Since all groups involved are monomial a Wedderburn component of
$F^\varphi H$ is determined by a pair $(S,T)$ of subgroups of $Y(H)$ which satisfy the following. $S$ has a linear character $\chi$ with kernel $T$ such that $\operatorname{ind}_S^{Y(H)}(\chi)$ is irreducible and $\chi$ has values on $x$, $y$, $z$, $\rho_1^q$, $\rho_2^q$, $\pi_1^{p^2}$ and $\pi_2^{p^2}$ which correspond to the powers of the natural generators $\alpha$, $\beta$, $\gamma$, $\kappa$, $\lambda$, $\mu$ and $\nu$ appearing in $\varphi$ respectively.
The corresponding matrix algebra lies over $F_4$ if
and only if $S$ contains an element of order $16$ none of whose
powers lies in $T$. So it is sufficient to show that for any
$\varphi \in H^2(H, F^*)$ there is a corresponding pair $(S,T)$ such that $S$ contains no element of order $16$.
Instead of describing $\varphi$ we will distinguish the different
$T$. For example the condition $x \in T$ means that in writing
$\varphi$ in the natural generators the factor $\alpha$ does not
appear. We will study some cases separately. Note that we can make
assumptions only on $\langle x,y,z,\rho_1^q, \rho_2^q,
\pi_1^{p^2}, \pi_2^{p^2}  \rangle \cap T$, since this fixes which
natural generators appear in $\varphi$.
The general goal in all cases will be to achieve
$\sigma_2 \in S \setminus T$, because then an element of
order $16$ does not commute with $S/T$, so there can be no element of order $16$ in $S$ which has no power in $T$. Set $Z = Z(Y(H)) = \langle x,\rho_2^q, y, \pi_1^2, z, \rho_1^q, \pi_2^2 \rangle$.

\begin{itemize}
\item[Case 1:] $x,z \in T$.\\
Let $S = \langle Z, \sigma_2, \rho_2, \tau_1, \rho_1, \sigma_1,
\pi_2 \rangle$. Then $S' = \langle \sigma_2^q, z, \tau_1, \sigma_1
\rangle$ and let $T$ be a subgroup of $S$ containing $S'$ such
that $S/T$ is cyclic and $T$ does not contain $\sigma_2$. Let
$\chi$ be a linear character of $S$ with kernel $T$. Then $\chi' =
\operatorname{ind}_S^{Y(H)} \chi$ is of degree $2$ and
$\chi'(\sigma_2) = \chi(\sigma_2) + \chi(\sigma_2)^{-1} \neq 2$.
Moreover $\chi'$ is irreducible, since otherwise it would
decompose into two linear characters. But linear characters
contain $\sigma_2$ in its kernel, since $\sigma_2 \in Y(H)'$, and
then we would have $\chi'(\sigma_2) = 2$.

\item[Case 2:] $x \notin T$, $z \in T$.\\
Set $S = \langle Z, \sigma_2, \rho_2, \tau_1, \sigma_1, \pi_2
\rangle$. Then $S' = \langle \sigma_2^q, z, \tau_1, \sigma_1
\rangle$. Let $T$ again be a subgroup of $S$ containing $S'$ such
that $S/T$ is cyclic, $\sigma_2 \notin T$ and let $\chi$ and
$\chi'$ be defined similarly as in Case 1. Then $\chi'$ is a
character of degree 10 such that $\chi'(\sigma_2) =
5(\chi(\sigma_2)+\chi(\sigma_2)^{-1})$. This means that the
restriction of $\chi'$ to $\langle \sigma_2, \pi_1 \rangle$ decomposes into five
$2$-dimensional characters. So if $\chi'$ decomposes it decomposes
into characters of even degree. But on the other hand its
restriction to $(\langle x \rangle \times \langle \rho_2 \rangle )
\rtimes \langle \rho_1 \rangle$ has to decompose into characters
of degree $5$, since these are the only characters of this group
not having $x$ in the kernel.

\item[Case 3:] $x \in T$, $z \notin T$.\\
Set $S = \langle Z, \sigma_2, \rho_2, \tau_1, \rho_1, \pi_2
\rangle$. Note that $\tau_1^{\pi_2} = z\tau_1^{-1} =
\tau_1(z\tau_1^3)$. So we have $S' = \langle \sigma_2^q, x,
z\tau_1^3 \rangle$. Again let $T$ be a normal subgroup of $S$ such
that $S/T$ is cyclic, $\sigma_2 \notin T$ and let $\chi$ and
$\chi'$ be defined as in the previous cases. If $\chi'$ decomposes
then the summands have even degree, due to the value of $\chi'$ on
$\sigma_2$ and its restriction to $\langle \sigma_2, \pi_1 \rangle$. But at the same time
the degree of a summand would be divisible by $5$, due to its character value $0$ on
$z$ and the character theory of the extraspecial $q$-group $\langle z, \tau_1, \sigma_1 \rangle $.

\item[Case 4:] $x,z \notin T$.\\
Set $S= \langle Z, \sigma_2, \tau_1, \rho_1, \pi_2 \rangle$. Then
$S' = \langle z\tau_1^3 \rangle$. Let again $T$, $\chi$ and
$\chi'$ have analogous properties as before such that $\sigma_2^q
\notin T$. Note that $\tau_1 \not \in T$. By Frobenius reciprocity and Clifford theory we have, considering the scalar product of characters,
$$\langle \chi' , \chi' \rangle_{Y(H)} = \sum_{g \in Y(H)/S} \langle \chi, \chi^g \rangle_S. $$
Now a system of coset representatives of $Y(H)/S$ is given by
$$\{\pi_1^i \rho_2^j \sigma_1^k \ | \ 0 \leq i \leq 1, \ 0 \leq j,k \leq q-1 \}.$$
Set $a_{i,j,k} = \pi_1^i \rho_2^j \sigma_1^k$. Then $\sigma_2^{a_{i,j,k}} = \sigma^{(1+qj)\cdot (-1)^i}$ and $\tau_1^{a_{i,j,k}} = z^k\tau_1$. Since $\langle \tau_1, \sigma_2 \rangle \cap T = 1$ we hence have $\chi^{a_{i,j,k}} = \chi$ if and only if $i = j = k = 0$. So $\chi'$ is irreducible.
\end{itemize}
\end{proof}

\begin{remark}\label{rem:OddCase}
The calculations of the cohomology groups for the groups $G$ and $H$ from Dade's example suggest
that if the groups are of odd order then it is very well possible
that $G \sim_F H$ over any field $F$. In the words of Passman, the
``surprise'' in the proof of Dade is the fact that $FG \cong FH$
for fields of characteristic $q$ and that ``this isomorphism is so
easily proved'' \cite[p. 664]{Passman}. This proof relies on the
fact that setting $e = \frac{1}{p}\sum_{i=0}^{p-1} \pi_1^{pi} $ in
$FG$ and $FH$ respectively, $FG$ and $FH$ are direct sums of
algebras isomorphic to $eFG$ and $eFH$ respectively. As $eFG \cong
F(G/\langle \pi_1^p \rangle) \cong F(H/\langle \pi_1^p \rangle)
\cong eFH$ the isomorphism of $FG$ and $FH$ follows immediately.

It seems impossible to imitate this argument in the twisted case,
since there is no natural idempotent in the twisted group ring of
a cyclic group corresponding to $e$. For example
$\mathbb{F}_5^\alpha C_4$ is a simple algebra isomorphic to
$\mathbb{F}_{5^4}$ for $[\alpha] \in H^2(G, \mathbb{F}_5^*)$ of
order $4$, so it has no quotients which ``kill'' exactly the
cyclic group of order $2$. This is a special instance of the fact
that a twisted group ring of $G$ has no ``obvious homomorphism''
\cite[p. 14]{Passman} to some twisted group ring of a given
quotient of $G$. So though $G \sim_F H$ might still be true for
any field $F$ the arguments to prove this would be different from
the argument of Dade.

Also Yamazaki covers can not bring the whole solution as $H^2(G,
F^*)$ can be infinite, e.g. for $F = \mathbb{Q}$, and
then no Yamazaki cover exists.
\end{remark}

\begin{remark}
The probably most famous example obtained in the study of the
classical group ring isomorphism problem is Hertweck's
counterexample to the integral isomorphism problem
\cite{Hertweck}. This counterexample consists of two
non-isomorphic groups $G$ and $H$ of order $2^{21}\cdot 97^{28}$
such that $\mathbb{Z}G \cong \mathbb{Z}H$. It is not clear to us
if there exists a ring $R$ such that $G \not\sim_R H$. But it is
clear that $RG\cong RH$ and $H^2(G,R^*)\cong H^2(H,R^*)$ for any
commutative ring $R$. This follows from the fact that $RG \cong R
\otimes_\mathbb{Z} \mathbb{Z}G$ and the functorial definition of
group cohomology, $H^n(G, M) \cong
\operatorname{Ext}^n_{\mathbb{Z}G}(\mathbb{Z}, G)$ for any
$G$-module $M$ and where $\operatorname{Ext}$ denotes the
$\operatorname{Ext}$-functor. So $H^2(G, M)$ depends only on the
group ring $\mathbb{Z}G$ and not $G$ itself. It would be very
interesting to determine if $G \sim_R H$ indeed holds
independently of $R$.
\end{remark}

\textbf{Acknowledgement:} We thank Yuval Ginosar for useful discussions.

\bibliographystyle{amsalpha}
\bibliography{TGRIP2}

\end{document}